\documentclass[12pt, reqno]{amsart}
\usepackage[utf8]{inputenc}
\usepackage{url}
\usepackage{amsmath,color}
\usepackage{amsfonts}
\usepackage{amssymb}
\usepackage{amsthm}
\usepackage{multirow}
\usepackage{graphicx} 
\usepackage{lipsum}

\newcommand{\kommentar}[1]{}
\renewcommand{\pmod}[1]{\,(\mathrm{mod}\,#1)}

\newcommand{\E}{\mathbb{E}}
\newcommand{\acom}[1]{{\color{blue}{Alexandra: #1}} }

\newtheorem{theorem}{Theorem}

\newtheorem{lemma}[theorem]{Lemma}
\newtheorem{corollary}[theorem]{Corollary}
\newtheorem{proposition}[theorem]{Proposition}

\numberwithin{equation}{section}
\numberwithin{theorem}{section}

\numberwithin{equation}{section}

\def\half{\tfrac{1}{2}}

\def\sumstar{\sideset{}{^*}\sum}

\newcommand{\fixmehide}[1]{}
\newcommand{\fixmelater}[1]{}
\newcommand{\fixmedone}[1]{}
\newcommand{\fixmehidden}[1]{}

\title[Simultaneous non-vanishing: Weighted central limit theorem]{Simultaneous non-vanishing of Dirichlet $L$--functions, II: Weighted central limit theorem}

\date{\today}


\begin{document}

\author{Hung M. Bui, Alexandra Florea and Micah B. Milinovich}
\address{Department of Mathematics, University of Manchester, Manchester M13 9PL, UK}
\email{hung.bui@manchester.ac.uk}
\address{UC Irvine, Mathematics Department, Rowland Hall, Irvine 92697, USA}
\email{floreaa@uci.edu}
\address{Department of Mathematics, University of Mississippi, University, MS 38677 USA}
\email{mbmilino@olemiss.edu}

\begin{abstract}
Under the Generalized Riemann Hypothesis, we prove a weighted central limit theorem for the joint distribution of four Dirichlet $L$--functions at the central point, twisted by the family of primitive characters to a large prime modulus. As an application, we show that a positive proportion of the characters in the family yield four central values that are simultaneously large, and a positive proportion yield values that are simultaneously nonzero and small.

\end{abstract}

\allowdisplaybreaks

\maketitle

\section{Introduction}

Central values of $L$--functions are fundamental objects in number theory and have been the focus of substantial work. It is expected that the vanishing of an $L$--function at the center of the critical strip (apart from ``trivial'' reasons, such as the root number being equal to $-1$) should encode significant arithmetic information about the object the $L$--function is attached to. For Dirichlet $L$--functions, Chowla \cite{chowla} conjectured that
$$L(\tfrac 12,\chi) \neq 0,$$
for every primitive quadratic Dirichlet character $\chi$. More generally, it is now expected that the same non-vanishing statement should hold for all primitive characters. Partial results in this direction go back to Balasubramanian and Murty \cite{BM}, who showed that, for a large prime modulus $q$, a positive proportion of characters $\chi \pmod q$ have a nonzero central value. Iwaniec and Sarnak \cite{IS1} subsequently proved that, for any $\varepsilon>0$, at least $1/3-\varepsilon$ of the primitive characters modulo $q$ satisfy $L(1/2,\chi) \neq 0$, provided $q$ is sufficiently large. This lower bound has since been sharpened in several works, including \cite{bui, KMN, QW}. 

Motivated by questions surrounding Landau–Siegel zeros, Iwaniec and Sarnak \cite{IS2} studied the simultaneous non-vanishing of the central values $L(1/2,f)$ and $L(1/2, f \otimes \chi_d)$, where $\chi_d$ denotes the quadratic character modulo $d$, and $f$ varies over an appropriate family of even newforms. They showed that the family average of $L(1/2,f) L(1/2,f \otimes \chi_d)$ is proportional to $L(1,\chi_d)$; hence to obtain a lower bound on $L(1,\chi_d)$ one would need to show that $L(1/2,f)$ and $ L(1/2,f \otimes \chi_d)$ are often simultaneously nonzero (however, this approach falls just short of obtaining a simultaneous non-vanishing result that would give the desired lower bound). 

Focusing on the family of Dirichlet $L$--functions, Zacharias \cite{Z} showed that, given three characters $\chi_1, \chi_2, \chi_3$ modulo $q$ for $q$ prime, a positive proportion of primitive characters $\chi \pmod q$ satisfy $L(1/2,\chi \chi_j) \neq 0$, for $j=1,2, 3$. In fact, Zacharias proved a slightly more refined result showing that a positive proportion of primitive characters $\chi \pmod q$ satisfy 
\begin{equation}
\label{lowerbound} 
| L ( \tfrac{1}{2},\chi \chi_j )| \geq \frac{1}{\log q},
\end{equation} for $j=1,2,3.$

In \cite{BFM}, the simultaneous non-vanishing of four Dirichlet $L$--functions at the central point is established. More precisely, for $\chi_1,\ldots,\chi_4$ even primitive Dirichlet characters modulo $D_1,\ldots, D_4$, respectively, where the $D_j$ are pairwise co-prime and square-free integers of size less than a small power of $q$, it is shown that under the Generalized Riemann Hypothesis (GRH), $\prod_{j=1}^4 L(1/2,\chi \chi_j) \neq 0$ for a positive proportion of Dirichlet characters $\chi \pmod q$, with $q$ prime and sufficiently large. By adapting the techniques in \cite{BFM}, one should be able to obtain the analogue of \eqref{lowerbound} in the case of four $L$--functions, under GRH. 

In this paper, we go significantly beyond the lower bound in \eqref{lowerbound} by proving that, for any fixed $c>0$, a positive proportion of primitive characters $\chi \pmod q$ satisfy
\begin{equation}
\label{target} |L(\tfrac12,\chi \chi_j)| \gg \exp \Big(c \sqrt{\log \log q} \Big),
\end{equation} simultaneously for $j=1,\ldots,4$. 

To do this, we prove a weighted central limit theorem for the joint distribution of four Dirichlet $L$--functions, drawing from the ideas introduced in \cite{BELP}. Studying central limit theorems for $L$--functions goes back to the work of Selberg \cite{selberg}, who showed that 
$$ \frac{1}{T} \text{meas} \bigg\{ t \in [T, 2T]: \frac{ \log |\zeta(\frac12+it)|}{\sqrt{\frac{1}{2} \log \log T}} \in (a,b) \bigg\}= \frac{1}{\sqrt{2 \pi } }\int_a^b e^{-u^2/2} \, du + O \Big(  \frac{ (\log \log \log  T)^2}{\sqrt{\log \log T}}\Big),$$ as $T \to \infty$. 

Similar results are expected to hold for families of $L$--functions. Keating and Snaith \cite{KS} have proposed analogous Selberg-type conjectures for the logarithms of central values of $L$--functions in various families. For example, one expects that as $\chi$ varies over primitive Dirichlet characters modulo $q$, $\log |L(1/2,\chi)|$ has a Gaussian distribution with mean $0$ and variance $\frac{1}{2} \log \log q$, as $q \to \infty$. Unlike in the case of the Riemann zeta-function, the Selberg analogues for families of $L$--functions remain open. Indeed, 
central values in the family may vanish, so $\log |L(1/2,\chi)|$ is not always defined, and current non-vanishing results are far from strong enough to remove this obstruction in full generality.

While the full Selberg results for families of $L$--functions are out of reach, one can establish one-sided central limit theorems. For example, Hough \cite{hough} adapted Selberg's method to study the distribution of central values in symplectic and orthogonal families and established that the distributions are asymptotically bounded above by Gaussian distributions with the appropriate mean and variance. Assuming both the Riemann Hypothesis and the Zero Density Hypothesis in these families, he also obtained the full normal law, confirming  the Keating--Snaith conjectures. In recent work, Radziwi\l\l\ and Soundararajan \cite{RS1, RS2} established both upper and lower bounds towards the normal distribution, with the upper bound constant equal to $1$, and the lower bound constant matching the known non-vanishing proportions in the families under consideration. 

A barrier in proving Selberg-type analogues lies, as previously noted, in the possibility of the central value vanishing. To overcome this barrier, Bui, Evans, Lester and Pratt \cite{BELP} introduced a weight which accounts for when the central value is zero, and established central limit theorems with respect to the weighted measure, for the family of primitive characters modulo $q$ and for the joint distribution of central values corresponding to twists of Hecke eigenforms. What the weighted approach yields, in contrast with the previous one-sided central limit theorems, is a genuine two-sided asymptotic for
the weighted measure, not merely an upper or lower bound. It is this two-sided
control that yields the simultaneous growth behavior in \eqref{target},
which one-sided bounds cannot provide.

In this paper, we build on the work in \cite{BELP} and combine it with ideas from \cite{BFM} to obtain a weighted central limit theorem for the joint distribution of four Dirichlet $L$--functions, which allows us to prove \eqref{target}.

 Throughout the paper, we assume that $q$ is a prime number for technical simplicity.  For $j=1,\ldots, 4$, let $D_j$ be pairwise co-prime with $1\leq D_j\leq D$, where $D= \max\{D_1,\ldots, D_4\}$. Let $\chi_j$ be a primitive Dirichlet character modulo $D_j$. For simplicity we shall assume that the $D_j$ are square-free and the $\chi_j$ are all even. We are interested in the simultaneous non-vanishing of the four Dirichlet $L$--functions $L(s,\chi\chi_j)$ at the center of the critical strip $s=1/2$ over the set of primitive characters $\chi$ modulo $q$.

Let $\varphi^{*}(q)$ denote the number of primitive characters modulo $q$. We write $\sum_{\chi \pmod q}^*$ to indicate that the summation is restricted to primitive characters, $\chi\ne\chi_0$. Given a complex-valued function $F$ on the set of primitive characters modulo $q$, we define 
\[
\varphi_F^{*}(q)=\sumstar_{\chi \pmod q} F(\chi).
\]
Let $\mu_F$ be the complex measure on the set of primitive characters modulo $q$ given by
\[
\mu_F(S)= \frac{1}{\varphi_F^{*}(q)}\sumstar_{\chi \in S} F(\chi), \qquad S \subset \{ \chi \pmod q\}.
\]
For example, if $F=1$, then $\varphi_F^*(q)=\varphi^{*}(q)$ and $\mu_F$ is the usual normalized counting measure. 
To account for the vanishing of the central $L$-values we will choose our weight $F$ so that $F(\chi)=0$ whenever any of the central values vanishes.
Moreover, 
to capture the typical behavior of the $L$--functions we would like $F \approx 1$ so as not to bias our measure. 
Our approach is to take 
\[
F(\chi)=\prod_{j=1}^2 L(\tfrac12,\chi\chi_j)M(\tfrac12,\chi\chi_j)\prod_{k=3}^4 L(\tfrac12,\overline{\chi\chi_k})M(\tfrac12,\overline{\chi\chi_k}),
\] 
where the mollifier $M$ dampens the extreme behavior of the central $L$--values. The mollifier is chosen so that
$$M (  \tfrac{1}{2},\chi) \approx \prod_{p \leq x} \Big( 1 - \frac{\chi(p)}{\sqrt{p}}\Big),
$$ for most $\chi$ in a ``good set'' of characters, where $x$ is a small power of $q$.

Following \cite{BFM}, the mollifier is defined as follows. We split the primes into intervals
$$I_0= (1, q_0], \, I_1=(q_0, q_1], \,\ldots, \, I_K= (q_{K-1}, q_K],$$ 
where $$q_k=q^{\beta_k}\qquad\text{and}\qquad \beta_k=\frac{e^k}{(\log\log q)^5},$$
with $\beta_K$ a small constant (see \eqref{condition_mol}). For $k \leq K$, let
$$a(p;k) = \Big(1 - \frac{\log p}{ \log q_k} \Big)p^{-\lambda/\log q_k},  $$ with $\lambda$ the unique solution to $e^{-\lambda} = \lambda+\lambda^2/2$.
We extend $a(p;k)$ to a completely multiplicative function in the first variable. 
Define
\begin{equation*}
M_k(\tfrac12,\chi ) := \sum_{\substack{p|n \Rightarrow p \in I_k \\ \Omega(n) \leq \ell_k}} \frac{ \chi(n) \mu(n) a(n;K) }{\sqrt{n}},
\end{equation*} for parameters $\ell_k$ which 
are defined as 
\begin{equation*}
    \ell_k = 2 \Big[ \frac{1}{2^{1/4}} \Big[\frac{1}{8\beta_k}\Big]^{3/4}\Big] .
\end{equation*}
Let
\begin{equation*}
M(\tfrac12,\chi ) = \prod_{k=0}^K M_k(\tfrac12,\chi).
\end{equation*}

It was proved in \cite[Theorem 1.4]{BFM} that
\begin{equation}\label{upperbd}
\sumstar_{\chi \pmod q}|L(\tfrac12,\chi\chi_j)M(\tfrac12,\chi\chi_j)|^k\ll_k q
\end{equation}
for every $k\in\mathbb{N}$, and that
\begin{equation}\label{weight}
\varphi_F^*(q)\asymp q.
\end{equation}

We will prove the following main theorem.
\begin{theorem}
\label{thm:cltjoint}
Assume GRH. Suppose that
\[
D^{272}\ll q^{11/16-1/2000}.
\]
For any intervals $U_j \subset \mathbb R$, with $j=1,\ldots, 4$, we have that
\[
\begin{split}
&\mu_F\Bigg(\bigg\{ \chi \pmod q, \chi\neq \chi_0 : \frac{\log |L(\tfrac12,\chi \chi_j)|}{\sqrt{\tfrac12 \log \log q}} \in U_j, j=1,\ldots, 4\bigg\}\Bigg)
\\
&\qquad = \frac{1}{4\pi^2} \int_{U_1 \times U_2\times U_3\times U_4} e^{-\sum x_j^2/2} \, dx_1dx_2dx_3dx_4+O_{\varepsilon}\left(  (\log \log q)^{-1/2+\varepsilon}\right).
\end{split}
\]
\end{theorem}

Theorem \ref{thm:cltjoint} yields the following corollary.
\begin{corollary} \label{cor:simul large vals}
	Assume GRH. Let $c>0$ be fixed. There are $\gg_{c} q$ characters $\chi$ modulo $q$ such that we simultaneously have
	\begin{align*}
		|L(\tfrac12, \chi\chi_j)| > \exp\Big(c \sqrt{\log\log q}\Big)
	\end{align*}
	for $j=1,\ldots, 4$. Additionally, there are $\gg_{c} q$ characters $\chi$ modulo $q$ such that we simultaneously have
	\begin{align*}
		0<|L(\tfrac12, \chi\chi_j)| < \exp\Big(-c \sqrt{\log\log q}\Big)
	\end{align*}
for $j=1,\ldots, 4$.
\end{corollary}

As in the two-dimensional central limit theorem of \cite{BELP}, the limiting density in Theorem \ref{thm:cltjoint}
factors: the four normalized central values
$\log|L(1/2,\chi\chi_j)|/\sqrt{\tfrac12\log\log q}$ become asymptotically
independent standard Gaussians, and it is this factorization that makes the
simultaneous statement of Corollary 1.2 possible. Asymptotic independence of distinct
twists is itself not new: in \cite{BELP} the two twists $L(1/2,f\otimes\chi)$ and
$L(1/2,g\otimes\chi)$ decorrelate because for $f\ne g$, the relevant cross-correlation
$\sum_{p \leq x} \lambda_f(p)\lambda_g(p)/p$ is $O(1)$ by the Prime Number Theorem for the
fixed Rankin--Selberg $L$--function $L(s,f\otimes g)$ (with $x$ being a small power of $q$).

Our work is different in two ways. First, the decorrelation is now governed not by a
fixed Rankin--Selberg factor but by the Dirichlet characters $\chi_j\overline{\chi_k}$,
whose modulus $D_jD_k$ grows with $q$; the cross-correlation
\[
\sum_{p\in I_0}\frac{\chi_j(p)\overline{\chi_k}(p)}{p} \ll \log\log\log q
\qquad(j\ne k)
\]
must therefore be controlled uniformly with a growing conductor, and the bound
one obtains is $\log\log\log q$ rather than the $O(1)$ in the fixed-form
setting of \cite{BELP}. Second, the statement is genuinely four-dimensional: joint
Gaussianity of all four values is a stronger structural fact than pairwise
independence, and requires the full six-term ``swap'' structure of the random model
(Section \ref{sectionrandom}) to cancel correctly, not merely the vanishing of pairwise correlation.

A crucial input in the proof is the asymptotic evaluation, established in \cite{BFM}, of the twisted
first moment of the fourfold product,
\[
\frac{1}{\varphi^*(q)}\ \sumstar_{\chi \pmod q}
\prod_{j=1}^2 L(\tfrac12,\chi\chi_j) \prod_{k=3}^4 L(\tfrac12, \overline{\chi\chi_k})\,\chi(\ell_1)\overline{\chi(\ell_2)},
\]
valid uniformly in the range $D^{272}L^{96}\ll q^{11/16-\varepsilon}$ (where $D= \max\{D_1,\ldots, D_4\}$ and $L= \max\{\ell_1,\ell_2\}$). The reason a twisted first moment is the relevant input is the mollification
philosophy of Radziwi\l\l\ and Soundararajan, adapted in \cite{BELP}. The mollifier $M$ is a short
Dirichlet polynomial mimicking $\prod_j L(1/2,\chi\chi_j)^{-1}$, engineered so that
the products $L(1/2,\chi\chi_j)M(1/2,\chi\chi_j)$ are close to $1$ for most
$\chi$, and hence $\log|L(1/2,\chi\chi_j)|$ is well-approximated by $-\log|M(1/2,\chi\chi_j)|$, which is essentially a short prime sum
$\Re\big(\sum_{p\in I_0}\chi\chi_j(p)/\sqrt p\big)$ (see \eqref{lm} and \eqref{approxM0}). The distribution of these prime sums, and
hence of the central values, is then extracted by expanding the characteristic
function in powers of the prime sums and evaluating
$\sum_\chi^*\prod_{j=1}^2 L(1/2,\chi\chi_j) \prod_{k=3}^4 L(1/2, \overline{\chi\chi_k})\,M\cdot(\text{prime
sum})^k$. Each such
quantity is a twisted first moment.

Another crucial input from \cite{BFM} is a sharp upper bound 
for high mollified moments, using ideas of Soundararajan \cite{S}, Harper \cite{H}, and Lester and Radziwi\l\l\ \cite{LR}. This is used to show that the exceptional set of
$\chi$ is negligible. 

The paper is organized as follows. In Section \ref{lemmas}, we gather a few lemmas that we will use throughout the paper, and state one of the main propositions (see Proposition \ref{mainprop}). In Section \ref{random}, we reduce the problem of computing mollified moments twisted by exponentials of prime sums 
to computing the analogous objects using a random model (see Corollary \ref{fchirandom}). In Section \ref{sectionrandom}, we prove key results which allow us to compute the mollified twisted moments on the random variable side. We prove the main propositions, Proposition \ref{main_prop} and Proposition \ref{mainprop}, in Section \ref{proofs}. Finally, we prove the main central limit theorem, Theorem \ref{thm:cltjoint}, in Section \ref{sec:proofcltjoint}, and Corollary \ref{cor:simul large vals} in Section \ref{cor_proof}.
\newline

\textbf{Acknowledgments:} AF was supported by the National Science Foundation (NSF CAREER grant DMS-2339274). MBM was supported in part by NSF grant DMS-2401461.

\section{Lemmas and key proposition}
\label{lemmas}
\begin{lemma}\label{exptrunc}
For $|x|\leq L/e^2$ we have
\begin{equation*}
\sum_{0\leq k\leq L} \frac{x^k}{k!}=e^{x}\big(1+O( e^{-L})\big).
\end{equation*}
\end{lemma}
\begin{proof}
See \cite[Lemma 1]{RS1}.
\end{proof}

For $j,k \leq K$, we let
$$P_{I_j}(\chi; k) = \Re\bigg(\sum_{p \in I_j} \frac{\chi(p) a(p;k)}{\sqrt{p}}\bigg)$$ and denote
\[
P(\chi)=P_{I_0}(\chi; K).
\]
 We shall need the following results.

\begin{lemma} \label{momentsP}
For $k\in\mathbb{N}$ and $\beta_0k \le 1-\varepsilon$ 
we have
\begin{equation*}
\sumstar_{\substack{\chi \pmod q}}  P(\chi\chi_j)^{2k} \le qk! \Big(\sum_{p\in I_0} \frac{1}{p} \Big)^k\leq qk!(\log\log q_0)^k.
\end{equation*}
\end{lemma}
\begin{proof}
We have
\begin{align}
    \label{powerp}
    P(\chi \chi_j)^{2k} &\leq \bigg|\sum_{p \in I_0} \frac{\chi\chi_j(p) a(p;K)}{\sqrt{p}}\bigg|^{2k}\nonumber\\
    &= (k!)^2\sum_{\substack{p |mn \Rightarrow p \in I_0 \\ \Omega(m)=\Omega(n)=k}} \frac{ \chi\chi_j(m) \overline{\chi \chi_j}(n) a(mn;K) \nu(m)\nu(n)}{\sqrt{mn}},
    \end{align}
    where $\nu(n)$ is a multiplicative function and $\nu(p^a)=1/a!$.
   Since $P(\chi \chi_j)^{2k}$ is nonnegative, we bound the sum over primitive $\chi$ by the sum over all $\chi \pmod q$. Also, since $m,n \leq q^{\beta_0k} <q$ we can isolate the diagonal terms with $m=n$ in \eqref{powerp} using orthogonality of characters.
 We hence obtain that 
    \begin{align*}
      \sumstar_{\substack{\chi \pmod q}}  P(\chi\chi_j)^{2k}   \leq  q (k!)^2  \sum_{\substack{p|n \Rightarrow p \in I_0 \\ \Omega(n) =k}} \frac{ a(n;K)^2 \nu(n)^2}{n}.
    \end{align*}
    Note that we have 
    \begin{align*}
        \sum_{\substack{p|n \Rightarrow p \in I_0 \\ \Omega(n) =k}} \frac{ a(n;K)^2 \nu(n)^2}{n} \leq \sum_{\substack{p|n \Rightarrow p \in I_0 \\ \Omega(n) =k}} \frac{\nu(n)}{n}=  \frac{1}{k!} \Big( \sum_{p \in I_0} \frac{1}{p} \Big)^k.
    \end{align*}
    Combining the two equations above, the conclusion follows.
\end{proof}

\begin{lemma} \label{lem:largedev}
For $V \ge 1$ and $\beta_0V^2\leq 3(1-\varepsilon)$ we have
\[
\# \bigg\{ \chi \pmod q,\chi\ne\chi_0 : |P(\chi\chi_j)| \ge V \sqrt{\tfrac12 \log \log q_0} \bigg\} \ll q e^{-V^2/3}.
\]
\end{lemma}
\begin{proof}
For $\beta_0k \le 1-\varepsilon$, using Chebyshev's inequality and Lemma \ref{momentsP} we have that 
\begin{align*}
&\# \left\{ \chi \pmod q,\chi\ne\chi_0 : |P(\chi\chi_j)| \ge V \sqrt{\tfrac12 \log \log q_0} \right\}\\
&\qquad\le \frac{1}{V^{2k} (\frac12 \log \log q_0)^k} \ \sumstar_{\substack{\chi \pmod q}} P(\chi\chi_j)^{2k}\leq q \frac{2^kk!}{ V^{2k}} \ll q \Big( \frac{3k}{eV^2}\Big)^k,
\end{align*}
by Stirling's formula. Taking $k=\lfloor V^2/3 \rfloor$, we obtain the lemma.
\end{proof}

The next proposition is the main input for Theorem \ref{thm:cltjoint}.

\begin{proposition}\label{mainprop}
For $u_j\ll 1$, with $j=1,\ldots,4$, we have
\begin{align*}
 & \frac{1}{\varphi_F^{*}(q)} \ \sumstar_{\chi \pmod q} F(\chi)  
\exp\bigg(  i\sum_{j=1}^4 u_jP(\chi\chi_j)\bigg) = \exp \Big(  - \frac{\sum_{j=1}^4 u_j^2}{4} \log \log q_0 \Big)  \\
&\qquad \times   \exp \bigg( O \Big(  \log \log \log q \sum_{j=1}^4 |u_j|^2\Big) \bigg) \bigg(1 +O \Big(  \sum_{j=1}^4 |u_j|\Big) \bigg) + O \big(  (\log q)^{-3}\big).
\end{align*}
\end{proposition}

We will prove Proposition \ref{mainprop} in Section \ref{sectionmainprop}.

\kommentar{\section{Proof of Proposition \ref{mainprop}}

\begin{lemma} \label{lem:matchrandom}
Assume GRH. Let $C>0$ be fixed and sufficiently large. Let $Z \ge 1$ and $L=30C Z^2\log \log q_0$, and suppose that $2C^2Z^2\beta_0\log\log q_0\leq 1$. Then for $u_j\in\mathbb{R}$ and $|u_j| \leq Z$ we have
\begin{align*}
&\sumstar_{\chi \pmod q} F(\chi)  
\exp\bigg(  i\sum_{j=1}^4 u_jP(\chi\chi_j)\bigg)\\
&\qquad=\sum_{0\le k\le L} \frac{i^k}{k!}\sum_{\substack{k_j\geq 0\\\sum k_j=k}} \binom{k}{k_1,k_2,k_3,k_4} u_1^{k_1}u_2^{k_2}u_3^{k_3}u_4^{k_4}\ \sumstar_{\chi \pmod q}F(\chi)  \prod_{j=1}^4 P(\chi\chi_j)^{k_j}+O\big(qe^{-L}\big) .
\end{align*}
\end{lemma}
\begin{proof}
Let
\begin{equation*}
\mathcal G= \left\{ \chi \pmod q,\chi\ne\chi_0 : |P(\chi\chi_j)|  \le C Z \log \log q_0,\, j=1,\ldots, 4 \right\}.
\end{equation*}
Then by Lemma \ref{lem:largedev} we obtain that
\[
\# \mathcal{G}^{c}\ll qe^{-CL/45}.
\]
By H\"older's inequality and \eqref{upperbd} we then get
\begin{equation}\label{badset}
\sumstar_{\chi \in \mathcal{G}^{c}} F(\chi)  
\exp\bigg(  i\sum_{j=1}^4 u_jP(\chi\chi_j)\bigg)\ll (\# \mathcal{G}^{c})^{1/2}\bigg(\ \,\sumstar_{\chi \pmod q}|F(\chi)|^2\bigg)^{1/2}\ll qe^{-CL/90}.
\end{equation}

For $\chi \in \mathcal G$, by Lemma \ref{exptrunc},  we have
\begin{equation*}
\exp\bigg(  i\sum_{j=1}^4 u_jP(\chi\chi_j)\bigg)=\sum_{0\le k\le L} \frac{i^k}{k!}\sum_{\substack{k_j\geq 0\\\sum k_j=k}} \binom{k}{k_1,k_2,k_3,k_4} \prod_{j=1}^4 u_j^{k_j}  P(\chi\chi_j)^{k_j}+O\left( e^{-L}\right).
\end{equation*}
Combining this with \eqref{badset} we obtain the lemma.
\end{proof}}

\section{ Reduction to a random model}
\label{random}
Here, we begin the preliminary steps towards the proof of Proposition \ref{mainprop}.

Let
\begin{equation*}
I(\ell_1,\ell_2)=\frac{1}{\varphi^*(q)}\ \,   \sumstar_{\chi \pmod  q} L(\tfrac12,\chi\chi_1)L(\tfrac12,\chi\chi_2)L(\tfrac12,\overline{\chi\chi_3})L(\tfrac12,\overline{\chi\chi_4})\chi(\ell_1)\overline{\chi}(\ell_2).
\end{equation*}
In \cite[Theorem 1.3]{BFM} the following theorem is proved\footnote{In \cite[Theorem 1.3]{BFM}, the theorem is established for even characters. One can easily modify the proof to handle odd characters to deduce Theorem \ref{thmBFM}.}.

\begin{theorem}\label{thmBFM}
Let $L=\max\{\ell_1,\ell_2\}$ and $\varepsilon>0$. Suppose that
\[
D^{272}L^{96}\ll q^{11/16-\varepsilon}.
\]
Then we have
\begin{align*}
I(\ell_1,\ell_2)&=M_{1,2,3,4}(\ell_1,\ell_2)+ \epsilon M_{3,4,1,2}(D_1D_2\ell_1,D_3D_4\ell_2)\\
&\qquad+\chi_1\overline{\chi_3}(q)\epsilon(\chi_1)\epsilon(\overline{\chi_3})M_{3,2,1,4}(D_1\ell_1,D_3\ell_2) +\chi_1\overline{\chi_4}(q)\epsilon(\chi_1)\epsilon(\overline{\chi_4})M_{4,2,3,1}(D_1\ell_1,D_4\ell_2)\\
&\qquad+\chi_2\overline{\chi_3}(q)\epsilon(\chi_2)\epsilon(\overline{\chi_3})M_{1,3,2,4}(D_2\ell_1,D_3\ell_2)+\chi_2\overline{\chi_4}(q)\epsilon(\chi_2)\epsilon(\overline{\chi_4})M_{1,4,3,2}(D_2\ell_1,D_4\ell_2)\\
&\qquad+O_\varepsilon\Big(q^{\varepsilon}\big(q^{-11/16}D^{272}L^{96}\big)^{1/28}\Big)+O_\varepsilon\Big(q^{\varepsilon}\big(q^{-25/32}D^{188}L^{50}\big)^{1/80}\Big),
\end{align*}
where
\begin{equation*}
\epsilon=\chi_1\chi_2\overline{\chi_3}\overline{\chi_4}(q)\epsilon(\chi_1)\epsilon(\chi_2)\epsilon(\overline{\chi_3})\epsilon(\overline{\chi_4}),
\end{equation*}
\[
M_{j_1,j_2,j_3,j_4}(\ell_1,\ell_2)=\sum_{\ell_1m=\ell_2n}\frac{(\chi_{j_1}*\chi_{j_2})(m)(\overline{\chi_{j_3}}*\overline{\chi_{j_4}})(n)}{\sqrt{mn}}V\Big(\frac{mn}{\widehat{q}^{\,2}}\Big),
\]
 and $V$ is given by\begin{equation}\label{formulaV+}
V(x)=\frac{1}{2\pi i}\int_{(1)}G(s)g(s)x^{-s}\frac{ds}{s},
\end{equation} with
 $G(s)$ an even entire function of rapid decay in any fixed strip $|\emph{Re}(s)| \leq C$ satisfying $G(0)= 1$, and
\[g(s)=\frac{1}{2}\bigg(\frac{\Gamma\big(\frac14  +\frac{s}{2}\big)^4}{\Gamma\big(\frac{1}{4}\big)^4}+\frac{\Gamma\big(\frac34  +\frac{s}{2}\big)^4}{\Gamma\big(\frac{3}{4}\big)^4}\bigg).
\]
\end{theorem}

Let $\{X(p)\}_p$ be i.i.d. uniformly distributed random variables on the unit circle. Define
$$X(n) = \prod_{p^a||n}X(p)^a.$$
Note that for $m,n < q$, we have
\begin{equation} \frac{1}{\varphi(q)} \sum_{\chi \pmod q} \chi(m) \overline{\chi(n)}= \E \Big( X(m) \overline{X(n)}\Big). \label{orthogonality}
\end{equation}
Let
\begin{align*}
L_V(X) &= L_{1,2,3,4}(X;V)+\epsilon X(D_1D_2) \overline{X(D_3D_4)} L_{3,4,1,2}(X;V) \\
&\qquad+ \chi_1\overline{\chi_3}(q)\epsilon(\chi_1)\epsilon(\overline{\chi_3}) X(D_1) \overline{X(D_3)} L_{3,2,1,4}(X;V) \\
&\qquad+\chi_1\overline{\chi_4}(q)\epsilon(\chi_1)\epsilon(\overline{\chi_4}) X(D_1) \overline{X(D_4)}L_{4,2,3,1}(X;V)\\
&\qquad+\chi_2\overline{\chi_3}(q)\epsilon(\chi_2)\epsilon(\overline{\chi_3})X(D_2) \overline{X(D_3)} L_{1,3,2,4}(X;V)\\
&\qquad+\chi_2\overline{\chi_4}(q)\epsilon(\chi_2)\epsilon(\overline{\chi_4})X(D_2) \overline{X(D_4)}L_{1,4,3,2}(X;V),
\end{align*}
where 
\begin{align*}
 &  L_{j_1,j_2,j_3,j_4}(X;V)  =   \sum_{m,n \geq 1}\frac{(\chi_{j_1}*\chi_{j_2})(m)(\overline{\chi_{j_3}}*\overline{\chi_{j_4}})(n)}{\sqrt{mn}}X(m) \overline{X(n)} V\Big(\frac{mn}{\widehat{q}^{\,2}}\Big).
\end{align*}
By Theorem \ref{thmBFM} we get that 
\begin{align}\label{consequenceBFM}
    I(\ell_1,\ell_2) &= \E \Big( L_V(X) X(\ell_1) \overline{X(\ell_2)}\Big) \\
&\qquad+O_\varepsilon\Big(q^{\varepsilon}\big(q^{-11/16}D^{272}L^{96}\big)^{1/28}\Big)+O_\varepsilon\Big(q^{\varepsilon}\big(q^{-25/32}D^{188}L^{50}\big)^{1/80}\Big).\nonumber
\end{align}

We also define for $j =1,\ldots, 4$,
$$P_{\chi_j}(X) =\Re\bigg(\sum_{p \in I_0} \frac{X(p) \chi_j(p) a(p;K)}{\sqrt{p}}\bigg) ,$$
and for $k \leq K$ and $j =1,\ldots,4$,
$$M_{k,\chi_j} (X)=\sum_{\substack{ p |n \Rightarrow p \in I_k \\ \Omega(n) \leq \ell_k}} \frac{ X(n) \chi_j(n) \mu(n) a(n;K) }{\sqrt{n}}.  $$
We let 
$$M_{\chi_j}(X) = \prod_{k=0}^K M_{k,\chi_j}(X), \qquad  M_k(X) = M_{k,\chi_1}(X) M_{k,\chi_2}(X) M_{k,\overline{\chi_3}}(\overline{X})M_{k,\overline{\chi_4}}(\overline{X}),$$
and $$M(X) = M_{\chi_1}(X) M_{\chi_2}(X) M_{\overline{\chi_3}}(\overline{X})M_{\overline{\chi_4}}(\overline{X}).$$
In light of Theorem \ref{thmBFM}, we impose the condition that 
\begin{equation} 
\label{condition_mol}
\beta_K<10^{-25}.
\end{equation}
   Let 
    $$F(X) = L_V(X) M(X).$$ 
Then the following lemma holds.
\begin{lemma}
\label{lem_expectation}
Let $k,k_j$, with $j=1,\ldots,4$, be nonnegative integers. Suppose that $$\beta_0 \sum_{j=1}^4 k_j < 10^{-7}.$$ Then there exists $\delta>0$ such that
\begin{equation}  \label{first} \frac{1}{\varphi^{*}(q)}\ \sumstar_{\chi \pmod q} F(\chi) \prod_{j=1}^4 P(\chi \chi_j)^{k_j}= \E \bigg( F(X) \prod_{j=1}^4 P_{\chi_j}(X)^{k_j}\bigg) + O ( q^{-\delta}).\end{equation}
  For $\beta_0k \leq 1-\varepsilon$ 
  we also have
  \begin{equation}\label{second} \E \Big(  P_{\chi_j}(X)^{2k}\Big) =\frac{1}{\varphi(q)} \sum_{\chi \pmod q} P(\chi \chi_j)^{2k}\leq k! \Big(\sum_{p \in I_0} \frac{1}{p} \Big)^k.
  \end{equation}
\end{lemma}
\begin{proof}
Note that $P(\chi \chi_j)$ is a Dirichlet polynomial of length $q_0$ for $j=1,\ldots,4$. Moreover, from \cite{BFM}, $M(\chi)$ is a short Dirichlet polynomial, which follows from condition \eqref{condition_mol}. Then \eqref{first} follows from \eqref{consequenceBFM} with an error term of size $q^{-\delta}$. 

Moreover, \eqref{second} easily follows from \eqref{orthogonality} and Lemma \ref{momentsP}.
\end{proof}

Let $C>0$ be fixed and sufficiently large (say, $C>120$). For $Z \geq 1$, let  
    $$ \mathcal{S} = \Big\{  \chi \pmod q: |P(\chi \chi_j)| \leq CZ \log \log q_0, \text{ for all } j=1,\ldots, 4\Big\}.$$

We will prove the following lemma.

\begin{lemma}
\label{random_chi}
Assume GRH. Suppose $1\leq Z \leq (\log\log q)^2/C$ and $u_j \in \mathbb{R}$ for $j=1,\ldots, 4$ with $|u_j| \leq Z$. Then we have 
\begin{align*}
   & \frac{1}{\varphi^{*}(q)}\ \sumstar_{\chi \in \mathcal{S}} F(\chi)  
\exp\bigg(  i\sum_{j=1}^4 u_jP(\chi\chi_j)\bigg) \\
&\qquad= \E \Bigg( F(X) \exp\bigg(  i\sum_{j=1}^4 u_jP_{\chi_j}(X)\bigg)\Bigg) +O\big(  (\log q)^{-2000}\big).
\end{align*}
\end{lemma}
\begin{proof}
    
Using Lemma \ref{exptrunc}, by taking $L= 30CZ^2 \log \log q_0$, for $\chi \in \mathcal{S}$ we have 
    \begin{align} \label{truncation}
    \exp\bigg(  i\sum_{j=1}^4 u_jP(\chi\chi_j)\bigg)& = \sum_{0 \leq k \leq L} \frac{i^k}{k!} \sum_{k_1+\ldots +k_4=k} {k \choose k_1,k_2,k_3,k_4} u_1^{k_1} u_2^{k_2} u_3^{k_3} u_4^{k_4} \prod_{j=1}^4 P(\chi \chi_j)^{k_j}\\
    &\qquad+O(e^{-L/2}).\nonumber
    \end{align}
 From H\"older's inequality and \eqref{upperbd}, we have
\begin{align*}
 \Big|\sumstar_{\chi \in \mathcal{S}^c} F(\chi)  \prod_{j=1}^4 P(\chi \chi_j)^{k_j} \Big| &\leq \Big( \sumstar_{\chi \in \mathcal{S}^c} |F(\chi)|^2 \Big)^{1/2} \prod_{j=1}^4 \Big(\ \sumstar_{\chi \pmod q} P(\chi \chi_j)^{8k_j} \Big)^{1/8}  \\
 & \leq q^{1/4} |\mathcal{S}^c|^{1/4} \prod_{j=1}^4 \Big(\ \sumstar_{\chi \pmod q} P(\chi \chi_j)^{8k_j} \Big)^{1/8}.
\end{align*}  
In view of Lemma \ref{momentsP} and Stirling's formula, the above is 
\begin{align*}
    & \ll q^{3/4} |\mathcal{S}^c|^{1/4}(\log\log q_0)^{k/2} \prod_{j=1}^4  \big((4k_j)!\big)^{1/8}  \ll q^{3/4} |\mathcal{S}^c|^{1/4} (4 \log \log q_0)^{k/2} \sqrt{ \prod_{j=1}^4k_j!}.
\end{align*}
Combining the equation above and \eqref{truncation}, we get that 
\begin{align*}
 &   \sumstar_{\chi \in \mathcal{S}}  F(\chi)  
\exp\bigg(  i\sum_{j=1}^4 u_jP(\chi\chi_j)\bigg) \\
&\quad= \sum_{0 \leq k \leq L} \frac{i^k}{k!} \sum_{k_1+\ldots +k_4=k} {k \choose k_1,k_2,k_3,k_4} u_1^{k_1} u_2^{k_2} u_3^{k_3} u_4^{k_4} \sumstar_{\chi \pmod q} F(\chi)  \prod_{j=1}^4 P(\chi \chi_j)^{k_j}\\
&\qquad + O \Bigg(  q e^{-L/2} + q^{3/4} |\mathcal{S}^c|^{1/4} \sum_{0 \leq k \leq L} (4 Z^2 \log \log q_0)^{k/2}\sum_{k_1+\ldots +k_4=k} \frac{1}{ \sqrt{ \prod_{j=1}^4 k_j!}}  \Bigg).
\end{align*}
From Lemma \ref{lem:largedev} with $V= CZ \sqrt{2 \log\log q_0}$, it follows that
\begin{equation}
\label{sc}
    |\mathcal{S}^c| \ll q \exp \Big(  -\frac{2 C^2 Z^2 \log \log q_0}{3} \Big).
    \end{equation} 
For simplicity of notation, let $x = 4Z^2 \log \log q_0$. We have
\begin{align}
 \label{sumj}   \sum_{0 \leq k \leq L}   (4 Z^2 \log \log q_0)^{k/2}  \sum_{k_1+\ldots +k_4=k} \frac{1}{ \sqrt{ \prod_{j=1}^4k_j!}} &=  \sum_{0 \leq k \leq L}   \sum_{k_1+\ldots +k_4=k}\prod_{j=1}^4 \frac{x^{k_j/2}}{ \sqrt{  k_j!}} \\
    & \leq \Big( \sum_{k \geq 0} \frac{x^{k/2}}{\sqrt{k!}}\Big)^4. \nonumber 
\end{align}
Now note that 
\begin{align*}
  \sum_{k \geq 0} \frac{x^{k/2}}{\sqrt{k!}} = \sum_{k \geq 0} \frac{(2x)^{k/2}}{\sqrt{k!}} \frac{1}{2^{k/2}}  \leq \Big(  \sum_{k \geq 0} \frac{(2x)^k}{k!}\Big)^{1/2} \Big( \sum_{k\geq 0} \frac{1}{2^k} \Big)^{1/2}  = \sqrt{2} e^x.
\end{align*}
It follows that
\begin{equation*}
    \eqref{sumj} \ll \exp \Big(  16 Z^2 \log \log q_0 \Big).
\end{equation*}
Combining the above and \eqref{sc} and using the fact that $C$ is large enough, we obtain that 
\begin{align*}
 \frac{1}{\varphi^{*}(q)} & \sumstar_{\chi \in \mathcal{S}}  F(\chi)  
\exp\bigg(  i\sum_{j=1}^4 u_jP(\chi\chi_j)\bigg)  = \sum_{0 \leq k \leq L} \frac{i^k}{k!} \sum_{k_1+\ldots +k_4=k} {k \choose k_1,k_2,k_3,k_4} u_1^{k_1} u_2^{k_2} u_3^{k_3} u_4^{k_4} \\
& \times \frac{1}{\varphi^{*}(q)}\ \sumstar_{\chi \pmod q} F(\chi)  \prod_{j=1}^4 P(\chi \chi_j)^{k_j} +O \big( (\log q)^{-2000} \big). 
\end{align*}
From Lemma \ref{lem_expectation} we get that
\begin{align*}
 \frac{1}{\varphi^{*}(q)} & \sumstar_{\chi \in \mathcal{S}}  F(\chi)  
\exp\bigg(  i\sum_{j=1}^4 u_jP(\chi\chi_j)\bigg)  = \sum_{0 \leq k \leq L} \frac{i^k}{k!} \sum_{k_1+\ldots +k_4=k} {k \choose k_1,k_2,k_3,k_4} u_1^{k_1} u_2^{k_2} u_3^{k_3} u_4^{k_4}\\
& \times \E\bigg(  F(X) \prod_{j=1}^4 P_{\chi_j}(X)^{k_j}\bigg)+ O \big( (\log q)^{-2000} \big). 
\end{align*}

By repeating the previous argument, but in the random variable setting, we have that
\begin{align*}
 & \sum_{0 \leq k \leq L} \frac{i^k}{k!} \sum_{k_1+\ldots +k_4=k} {k \choose k_1,k_2,k_3,k_4} u_1^{k_1} u_2^{k_2} u_3^{k_3} u_4^{k_4} \, \E\bigg(  F(X) \prod_{j=1}^4 P_{\chi_j}(X)^{k_j}\bigg)\\
  &\qquad= \E \Bigg( F(X) \exp\bigg(  i\sum_{j=1}^4 u_jP_{\chi_j}(X)\bigg) \Bigg) + O \big((\log q)^{-2000}\big).
\end{align*}
Combining the last two equations above finishes the proof of the lemma.
\end{proof}


We deduce the following corollary.

\begin{corollary}
\label{fchirandom}
    Under the same assumptions as in Lemma \ref{random_chi} we have 
\begin{align*}
    &\frac{1}{\varphi^{*}(q)}\ \sumstar_{\chi \pmod q} F(\chi)  
\exp\bigg(  i\sum_{j=1}^4 u_jP(\chi\chi_j)\bigg)\\
&\qquad= \E \Bigg( F(X) \exp\bigg(  i\sum_{j=1}^4 u_jP_{\chi_j}(X)\bigg)\Bigg)+O\big(  (\log q)^{-2000}\big).
\end{align*}
\end{corollary}

\begin{proof}
    We write
    \begin{align}
        &\sumstar_{\chi \pmod q}  F(\chi)  
\exp\bigg(  i\sum_{j=1}^4 u_jP(\chi\chi_j)\bigg)  \nonumber  \\
&\qquad= \sumstar_{\chi \in \mathcal{S}}  F(\chi)  
\exp\bigg(  i\sum_{j=1}^4 u_jP(\chi\chi_j)\bigg) + \sumstar_{\chi \in \mathcal{S}^c}  F(\chi)  
\exp\bigg(  i\sum_{j=1}^4 u_jP(\chi\chi_j)\bigg).\label{all_characters}
    \end{align}
    Note that by the Cauchy-Schwarz inequality, 
    \begin{align*}
     \sumstar_{\chi \in \mathcal{S}^c}  F(\chi)  
\exp\bigg(  i\sum_{j=1}^4 u_jP(\chi\chi_j)\bigg) \ll \bigg(\  \sumstar_{\chi \pmod q} |F(\chi)|^2\bigg)^{1/2} |\mathcal{S}^c|^{1/2}.   
    \end{align*}
    By \eqref{upperbd} and \eqref{sc}, it follows that
    $$ 
     \sumstar_{\chi \in \mathcal{S}^c}  F(\chi)  
\exp\bigg(  i\sum_{j=1}^4 u_jP(\chi\chi_j)\bigg) \ll  q (\log q)^{-2000}. $$
Combining this with \eqref{all_characters} and Lemma \ref{random_chi}, the conclusion follows.
\end{proof}

\section{Computations using the random model}\label{sectionrandom}
We define for $\Re(s)>1$ and $\{j_1,\ldots, j_4\}=\{1,\ldots,4\}$,
\begin{align*}
    L_{j_1,j_2,j_3,j_4}(s,X) &= \sum_{m,n \geq 1}\frac{(\chi_{j_1}*\chi_{j_2})(m)(\overline{\chi_{j_3}}*\overline{\chi_{j_4}})(n)}{(mn)^{s}}X(m) \overline{X(n)} \\
    & = \prod_p \Big( 1- \frac{X(p) \chi_{j_1}(p)}{p^s}\Big)^{-1} \Big( 1- \frac{X(p) \chi_{j_2}(p)}{p^s}\Big)^{-1} \\
    &\qquad\qquad\times\Big( 1- \frac{\overline{X(p) \chi_{j_3}(p)}}{p^s}\Big)^{-1}\Big( 1- \frac{\overline{X(p) \chi_{j_4}(p)}}{p^s}\Big)^{-1}\\
    &:= \prod_p L_{p,j_1,j_2,j_3,j_4}(s,X).
\end{align*}
Also define
\begin{align}
    L(s,X) &=  L_{1,2,3,4}(s,X)+\epsilon X(D_1D_2) \overline{X(D_3D_4)} L_{3,4,1,2}(s,X) \label{lx} \\
&\qquad+ \chi_1\overline{\chi_3}(q)\epsilon(\chi_1)\epsilon(\overline{\chi_3}) X(D_1) \overline{X(D_3)} L_{3,2,1,4}(s,X) \nonumber \\
&\qquad+\chi_1\overline{\chi_4}(q)\epsilon(\chi_1)\epsilon(\overline{\chi_4}) X(D_1) \overline{X(D_4)}L_{4,2,3,1}(s,X) \nonumber \\
&\qquad+\chi_2\overline{\chi_3}(q)\epsilon(\chi_2)\epsilon(\overline{\chi_3})X(D_2) \overline{X(D_3)} L_{1,3,2,4}(s,X) \nonumber \\
&\qquad+\chi_2\overline{\chi_4}(q)\epsilon(\chi_2)\epsilon(\overline{\chi_4})X(D_2) \overline{X(D_4)}L_{1,4,3,2}(s,X).\nonumber
\end{align}

Then we will prove the following proposition.

\begin{proposition}
\label{main_prop}
    The function $\E\big(L(s,X) M(X)\big)$ can be analytically continued to $\Re(s)>-C$, for any $C>0$. Moreover, for $u_j\ll1$, with $j=1,\ldots,4$, we have
    \begin{align*}
        &\E\Bigg( L_V(X) M(X) \exp\bigg(  i\sum_{j=1}^4 u_jP_{\chi_j}(X)\bigg)\Bigg)= \E \Big(L(\tfrac12,X)M(X)\Big) \exp \Big(  - \frac{\sum_{j=1}^4 u_j^2}{4} \log \log q_0 \Big)  \\
    &\qquad\qquad \times \exp \bigg( O \Big(  \log \log \log q \sum_{j=1}^4 |u_j|^2\Big) \bigg) \bigg(1 +O \Big(  \sum_{j=1}^4 |u_j|\Big) \bigg) +O\big(  (\log q)^{-3}\big).
    \end{align*}
\end{proposition}

   We begin the proof of the proposition by using the expression for $V$ in \eqref{formulaV+}, getting that 
    \begin{align}
    \label{mellin}
        \E\Bigg( L_V(X) M(X) \exp\bigg(  i\sum_{j=1}^4 u_jP_{\chi_j}(X)\bigg)\Bigg)  = \frac{1}{2 \pi i} \int_{(1)} G(s)g(s) \widehat{q}^{\,2s} F(s;\textbf{u}) \, \frac{ds}{s},
    \end{align}
    where
    $$F(s;\textbf{u}) = \E \Bigg(L(\tfrac12+s,X)M(X) \exp\bigg(  i\sum_{j=1}^4 u_jP_{\chi_j}(X)\bigg) \Bigg).$$
    
    We will later shift the line of integration to $\Re(s)=-\frac{(\log \log q)^2}{\log q}$ to pick up a simple pole at $s=0$. In the remainder of the section, we will prove the following auxiliary lemmas evaluating the contributions of different primes to the function $F(s;\textbf{u})$ in this region. We will finish the proof of the proposition in Section \ref{proofmainprop}.
    
    \subsection{Big primes}

    \begin{lemma}
    \label{bigp}
    For $\{j_1, \ldots, j_4\} = \{1,\ldots,4\}$ and $p>q_K$ we have 
   \begin{align*}
       &\E \Big( L_{p,j_1,j_2,j_3,j_4} ( \tfrac{1}{2}+s,X)\Big)\\
       &\qquad=L_p(1+2s, \chi_{j_1} \overline{\chi_{j_3}}) L_p(1+2s, \chi_{j_1} \overline{\chi_{j_4}}) L_p(1+2s, \chi_{j_2} \overline{\chi_{j_3}}) L_p(1+2s, \chi_{j_2} \overline{\chi_{j_4}}). 
       \end{align*} 
    \end{lemma}
    \begin{proof}
        We rewrite
        $$L_{p,j_1,j_2,j_3,j_4} ( \tfrac{1}{2}+s,X) = \sum_{a,b \geq 0} \frac{(\chi_{j_1}*\chi_{j_2})(p^{a})(\overline{\chi_{j_3}}*\overline{\chi_{j_4}})(p^{b})}{p^{(a+b)(1/2+s)}}X(p^{a}) \overline{X(p^{b})}.$$
        Taking the expected value in the equation above leads to $a=b$, and we get that
        \begin{align*}
        &\E \Big(L_{p,j_1,j_2,j_3,j_4} ( \tfrac{1}{2}+s,X) \Big)  = \sum_{a \geq 0} \frac{(\chi_{j_1}*\chi_{j_2})(p^{a})(\overline{\chi_{j_3}}*\overline{\chi_{j_4}})(p^{a})}{p^{a(1+2s)}}\\
        &\qquad= L_p(1+2s, \chi_{j_1} \overline{\chi_{j_3}}) L_p(1+2s, \chi_{j_1} \overline{\chi_{j_4}}) L_p(1+2s, \chi_{j_2} \overline{\chi_{j_3}}) L_p(1+2s, \chi_{j_2} \overline{\chi_{j_4}}),\end{align*}
    as claimed.
    \end{proof}
The lemma above leads to the following result.
\begin{corollary}
    The function $F(s;\textbf{u})$ can be analytically continued to $\Re(s)>-C$, for any $C >0$.
\end{corollary}
\begin{proof}
    Using the dominated convergence theorem we have that for $\Re(s)>1/2,$
    $$ \E \Bigg(  \prod_{p>q_K} L_{p,j_1,j_2,j_3,j_4}( \tfrac{1}{2}+s,X)\Bigg)= \prod_{p> q_K} \E \Big( L_{p,j_1,j_2,j_3,j_4} (  \tfrac{1}{2}+s,X)\Big).$$
    From Lemma \ref{bigp} and for $\Re(s)>1/2$, we get that
 \begin{align}\label{exp}
 & \E \Bigg(  \prod_{p>q_K} L_{p,j_1,j_2,j_3,j_4}( \tfrac{1}{2}+s,X )\Bigg) \nonumber \\
 &\ =L(1+2s,\chi_{j_1}\overline{\chi_{j_3}}) L(1+2s,\chi_{j_1}\overline{\chi_{j_4}}) L(1+2s,\chi_{j_2}\overline{\chi_{j_3}})L(1+2s,\chi_{j_2}\overline{\chi_{j_4}})  \\
  &\quad \times  \prod_{p \leq q_K} L_p(1+2s, \chi_{j_1} \overline{\chi_{j_3}})^{-1} L_p(1+2s, \chi_{j_1} \overline{\chi_{j_4}})^{-1} L_p(1+2s, \chi_{j_2} \overline{\chi_{j_3}})^{-1} L_p(1+2s, \chi_{j_2} \overline{\chi_{j_4}})^{-1}. \nonumber
 \end{align}   
 This furnishes an analytic continuation for any $s.$ 
\end{proof}
We also have the following.
\begin{lemma}\label{bigprimes}
   Assume GRH. For $-\frac{(\log \log q)^2}{\log q} \ll  \Re(s) \leq 2$ and $|\Im(s)|\leq q$, we have 
    $$ \E \Bigg(  \prod_{p>q_K} L_{p,j_1,j_2,j_3,j_4}\big( \tfrac{1}{2}+s,X \big)\Bigg)  = 1+ O ( q^{-\delta}),  $$
    for some $\delta>0$.
\end{lemma}
\begin{proof}
    Using Lemma $8.2$ in \cite{gs} we have that, for $\Re(s)\gg -\frac{(\log \log q)^2}{\log q}$ and $|\Im(s)| \leq q$, under GRH, 
    \begin{equation} \label{lemma_soundgranv}
         \log L(1+2s,\chi) = \sum_{n\leq q_K} \frac{\Lambda(n) \chi(n)}{n^{1+2s} \log n} + O \Big(  \frac{ \log q}{(\sigma_1 - 1/2)^2} q_K^{\sigma_1 - 1-2\Re(s)}\Big),
         \end{equation} where
    $$\sigma_1 = \min \Big\{ \frac{ 2\Re(s)+3/2}{2}, \frac{1}{2} + \frac{1}{ \log q_K} \Big\}= \frac{1}{2}+\frac{1}{ \log q_K}.$$
    Now for $\chi \in \{ \chi_{j_1} \overline{\chi_{j_3}}, \chi_{j_1} \overline{\chi_{j_4}}, \chi_{j_2} \overline{\chi_{j_3}},\chi_{j_2} \overline{\chi_{j_4}}\}, $ using equation \eqref{lemma_soundgranv}, we have that 
\begin{align*}
   & \log L(1+2s,\chi)  - \sum_{p \leq q_K} \log L_p(1+2s,\chi)\\
   &\qquad= O \bigg( \sum_{\substack{ p \leq q_K, r \geq 2 \\ p^r>q_K}} \frac{\chi(p)^r}{p^{r(1+2s)}r}\bigg)+O \Big( (\log q)^3 q_K^{-1/2} \exp \Big( 2 \beta_K (\log \log q)^2\Big)\Big) = O ( q^{-\delta}),
\end{align*}
    for some $\delta>0$. Exponentiating the above expression and using equation \eqref{exp}, we obtain the conclusion.
\end{proof}

We will also evaluate the contribution from the medium primes.

\subsection{Medium primes}
\begin{lemma}
\label{medium}
  For $-\frac{(\log \log q)^2}{\log q} \ll  \Re(s) \leq 2$,  we have
  \begin{align*}
      & \E \Bigg( \prod_{q_0 <p \leq q_K} L_{p,j_1,j_2,j_3,j_4} ( \tfrac{1}{2}+s,X ) \prod_{k=1}^K M_{k,\chi_1}(X)M_{k,\chi_2}(X) M_{k,\overline{\chi_3}}(\overline{X})M_{k,\overline{\chi_4}}(\overline{X}) \Bigg) \\
  &\qquad\ll (\log \log q)^{O(1)}.
   \end{align*}
\end{lemma}
\begin{proof}
    Let $$\gamma_{\chi_j}(n) = \chi_j(n) \mu(n) a(n;K) . $$ We have that 
    \begin{align}
        \E  & \Bigg(   \prod_{p \in I_k} L_{p,j_1,j_2,j_3,j_4}( \tfrac{1}{2}+s,X)  M_{k,\chi_1}(X)M_{k,\chi_2}(X) M_{k,\overline{\chi_3}}(\overline{X})M_{k,\overline{\chi_4}}(\overline{X}) \Bigg) \nonumber \\
        &= \E \Bigg( \sum_{p|m_1m_2 \Rightarrow p \in I_k} \frac{ (\chi_{j_1} * \chi_{j_2})(m_1) (\overline{\chi_{j_3}} * \overline{\chi_{j_4}})(m_2)}{(m_1m_2)^{1/2+s}} X(m_1) \overline{X(m_2)}  \nonumber \\
        & \qquad\qquad\times \sum_{\substack{p |n_1n_2n_3n_4 \Rightarrow p \in I_k \\ \Omega(n_j) \leq \ell_k, j=1,\ldots,4}} \frac{\gamma_{\chi_{1}}(n_1) \gamma_{\chi_{2}}(n_2) \gamma_{\overline{\chi_{3}}}(n_3) \gamma_{\overline{\chi_{4}}}(n_4)}{\sqrt{n_1n_2n_3n_4}} X(n_1n_2) \overline{X(n_3n_4)}\Bigg) \nonumber \\
        &= \sum_{\substack{ p |m_1m_2n_1n_2n_3n_4 \Rightarrow p \in I_k \\ m_1n_1n_2=m_2n_3n_4 \\ \Omega(n_j) \leq \ell_k, j=1,\ldots,4}} \frac{ (\chi_{j_1} * \chi_{j_2})(m_1) (\overline{\chi_{j_3}} * \overline{\chi_{j_4}})(m_2) \gamma_{\chi_1}(n_1) \gamma_{\chi_2}(n_2) \gamma_{\overline{\chi_3}}(n_3) \gamma_{\overline{\chi_4}}(n_4) }{\sqrt{m_1m_2n_1n_2n_3n_4} (m_1m_2)^s}.\label{exp2}
    \end{align}
    For ease of notation, let $a(m) = (\chi_{j_1} *\chi_{j_2})(m)m^{-s}$, $b(m) = (\overline{\chi_{j_3}}*\overline{\chi_{j_4}})(m) m^{-s}.$
    We use the fact that
    $$ \mathbf{1}_{\Omega(n) \leq \ell}= \frac{1}{2 \pi i } \int_{|z|=r} z^{\Omega(n)} \frac{1-z^{\ell+1}}{(1-z)z^{\ell+1}} \, dz,$$
    and then we get
    \begin{equation} \label{integral} \eqref{exp2} = \Big( \frac{1}{2 \pi i} \Big)^4 \int \int \int \int \mathcal{F}(z_1,z_2,z_3,z_4) \prod_{j=1}^4 \frac{1-z_j^{\ell_k+1}}{(1-z_j)z_j^{\ell_k+1}} \, dz_1 \, dz_2 \, dz_3 \, dz_4, 
    \end{equation} where we are integrating along $|z_j|=r,$ for $j=1,\ldots,4,$ and $1<r \leq 2$, and where
    \begin{align*}
        \mathcal{F}  (z_1,z_2,z_3,z_4)&= \sum_{\substack{ p |n_1n_2n_3n_4m_1m_2 \Rightarrow p \in I_k \\ n_1n_2m_1=n_3n_4m_2 }} z_1^{\Omega(n_1)} z_2^{\Omega(n_2)} z_3^{\Omega(n_3)} z_4^{\Omega(n_4)} \\
        & \qquad\qquad\times \frac{ a(m_1) b(m_2)\gamma_{\chi_1}(n_1) \gamma_{\chi_2}(n_2) \gamma_{\overline{\chi_3}}(n_3) \gamma_{\overline{\chi_4}}(n_4) }{\sqrt{m_1m_2n_1n_2n_3n_4}}.
    \end{align*}
    By multiplicativity, we have 
    \begin{align*}
     \mathcal{F} (z_1,z_2,z_3,z_4)&=  \prod_{p \in I_k} \bigg(1 + \frac{1}{p} \Big(  z_1 z_3 \gamma_{\chi_1}(p) \gamma_{\overline{\chi_3}}(p) + z_1 z_4 \gamma_{\chi_1}(p) \gamma_{\overline{\chi_4}}(p) +z_1 b(p) \gamma_{\chi_1}(p)  \\
     &\qquad + z_2z_3\gamma_{\chi_2}(p) \gamma_{\overline{\chi_3}}(p) + z_2 z_4 \gamma_{\chi_2}(p) \gamma_{\overline{\chi_4}}(p) + z_2b(p) \gamma_{\chi_2}(p)\\
     &\qquad + z_3a(p) \gamma_{\overline{\chi_3}}(p)+z_4a(p) \gamma_{\overline{\chi_4}}(p)  +a(p)b(p) \Big)  +O \Big( \frac{1}{p^2} + \frac{1}{p^{2+4\Re(s)}}\Big)\bigg)   . 
    \end{align*}
    Recall that $-\frac{(\log \log q)^2}{\log q} \ll \Re(s) \leq 2.$ Let $\alpha(p)$ denote the coefficient of $1/p$ in the Euler product above. Note that
    $$ \sum_{p \in I_k} \frac{\alpha(p)}{p} = O(1),$$
    and 
    $$ \sum_{p \in I_k } \frac{1}{p^{2+4\Re(s)}} \ll \exp\bigg(  4 (\log \log q)^2-  \frac{\log q}{(\log \log q)^{5}}\bigg) \ll \exp\bigg( -  \frac{\log q}{2(\log \log q)^{5}}\bigg).$$
    It follows that $|\mathcal{F}(z_1,z_2,z_3,z_4)| \ll 1$ uniformly for $|z_j| \ll 1$, and then using \eqref{integral}, we have that
    $$ \eqref{exp2} \ll 1.$$
    Using this for each $1 \leq k \leq K$, noting that $K\asymp \log\log\log q$, the conclusion follows. 
\end{proof}

\subsection{Small primes}
Let $$\widetilde{M_0}(X) = \widetilde{M_{0,\chi_1}}(X)\widetilde{M_{0,\chi_2}}(X) \widetilde{M_{0,\overline{\chi_3}}}(\overline{X})\widetilde{M_{0,\overline{\chi_4}}}(\overline{X}),$$
where $$\widetilde{M}_{0, \chi_j}(X)  = \sum_{\substack{ p |n \Rightarrow p \in I_0 \\ }} \frac{ X(n) \chi_j(n) \mu(n) a(n;K) }{\sqrt{n}}.$$
Write 
$$\widetilde{M}_0(X) = \prod_{p \in I_0} \widetilde{M_p}(X),$$
where
\begin{align*} \widetilde{M_p}(X) &= \Big(1 - \frac{X(p) \chi_1(p) a(p;K)}{\sqrt{p}} \Big)\Big(1 - \frac{X(p) \chi_2(p) a(p;K)}{\sqrt{p}} \Big) \\
&\qquad \times  \Big(1 - \frac{\overline{X(p) \chi_3(p)} a(p;K)}{\sqrt{p}} \Big)\Big(1 - \frac{ \overline{X(p) \chi_4(p)} a(p;K)}{\sqrt{p}} \Big).
\end{align*}
\begin{lemma}
\label{no_restriction}
For $\Re(s) \gg - \frac{ (\log \log q)^2}{\log q}$ we have
\begin{align*}
    &\E   \Bigg(  \prod_{p \in I_0} L_{p,j_1,j_2,j_3,j_4} ( \tfrac{1}{2}+s,X ) M_0(X)  \exp\bigg(  i\sum_{j=1}^4 u_jP_{\chi_j}(X)\bigg) \Bigg) \\
    &\  = \E \Bigg(  \prod_{p \in I_0} L_{p,j_1,j_2,j_3,j_4} ( \tfrac{1}{2}+s,X ) \widetilde{M_0}(X)  \exp\bigg(  i\sum_{j=1}^4 u_jP_{\chi_j}(X)\bigg) \Bigg)+ O\Big(  \exp \big(-   (\log \log q)^3\big)\Big).
\end{align*}
\end{lemma}
\begin{proof}
 Let $$R(X) = \widetilde{M_0}(X)-M_0(X).$$ We write
 \begin{align*}
     M_0(X) = \sum_{p | mn \Rightarrow p \in I_0} \frac{X(m) \overline{X(n)} a(mn;K) \alpha(m;\ell_0) \beta(n;\ell_0)}{\sqrt{mn}},
 \end{align*}
 where
 $$ \alpha(m;\ell_0) = \sum_{\substack{m=m_1 m_2 \\ \Omega(m_1),\Omega(m_2) \leq \ell_0}} \chi_1(m_1) \chi_2(m_2) \mu(m_1) \mu(m_2),$$ and 
 $$ \beta(n;\ell_0)= \sum_{\substack{n=n_1 n_2 \\ \Omega(n_1),\Omega(n_2) \leq \ell_0}} \overline{\chi_3}(n_1) \overline{\chi_4}(n_2) \mu(n_1) \mu(n_2).$$
 We similarly have
 \begin{align*}
     \widetilde{M_0}(X) = \sum_{p | mn \Rightarrow p \in I_0} \frac{X(m) \overline{X(n)} a(mn;K) \alpha(m) \beta(n)}{\sqrt{mn}},
 \end{align*}
 where $\alpha(m)$ and $\beta(n)$ are defined similarly to $\alpha(m;\ell_0)$ and $\beta(n;\ell_0)$, without the restrictions on the number of prime factors. Then
 $$R(X) = \sum_{\substack{p | m_1m_2n_1n_2 \Rightarrow p \in I_0 \\ \max\{\Omega(m_1), \Omega(m_2),\Omega(n_1),\Omega(n_2)\}>\ell_0}} \frac{X(m_1m_2) \overline{X(n_1n_2)} a(m_1m_2n_1n_2;K) \alpha(m_1m_2) \beta(n_1n_2)}{\sqrt{m_1m_2n_1n_2}}.$$

 Now note that
 \begin{align}
    &\Bigg| \E    \Bigg(  \prod_{p \in I_0} L_{p,j_1,j_2,j_3,j_4} ( \tfrac{1}{2}+s,X ) R(X)  \exp\bigg(  i\sum_{j=1}^4 u_jP_{\chi_j}(X)\bigg) \Bigg) \Bigg| \nonumber\\
    &\qquad\leq  \E \Bigg(  \bigg| \prod_{p \in I_0} L_{p,j_1,j_2,j_3,j_4} ( \tfrac{1}{2}+s,X ) R(X)  \bigg| \Bigg) \nonumber \\
    &\qquad \leq \E \Bigg( \bigg| \prod_{p \in I_0} L_{p,j_1,j_2,j_3,j_4} ( \tfrac{1}{2}+s,X )\bigg|^2\Bigg)^{1/2} \E \Big(  | R(X)|^2\Big)^{1/2} .\label{cs}
 \end{align}
 We first bound $\E \big(  | R(X)|^2\big)$. We have 
 \begin{align*}
     \E \Big( | R(X)|^2\Big) &= \sum_{\substack{p|m_1m_2n_1n_2m_1'm_2'n_1'n_2' \Rightarrow p \in I_0 \\ \max\{\Omega(m_1), \Omega(m_2),\Omega(n_1),\Omega(n_2)\}>\ell_0\\\max\{\Omega(m_1'), \Omega(m_2'),\Omega(n_1'),\Omega(n_2')\}>\ell_0}}  \E \Big( X(m_1m_2n_1'n_2') \overline{X(m_1'm_2'n_1n_2) }\Big) \\
     &\qquad \times \frac{a(m_1m_2n_1n_2m_1'm_2'n_1'n_2';K) \alpha(m_1m_2) \beta(n_1n_2) \overline{\alpha(m_1'm_2') \beta(n_1'n_2')}}{\sqrt{m_1m_2n_1n_2m_1'm_2'n_1'n_2'}}.
 \end{align*}
 Note that the condition $$\max\{\Omega(m_1), \Omega(m_2),\Omega(n_1),\Omega(n_2)\},\max\{\Omega(m_1'), \Omega(m_2'),\Omega(n_1'),\Omega(n_2')\}>\ell_0$$ implies that $2^{\ell_0}< 2^{\Omega(m_1m_2n_1n_2m_1'm_2'n_1'n_2')}.$ Taking the expected value, we isolate the terms with $m_1m_2n_1'n_2'=m_1'm_2'n_1n_2=:r$. Using the fact that $|\alpha(m_1m_2)|\leq d(r)$ and the similar inequalities for the other coefficients, we then get that
 \begin{align}
 \label{exp_r}
   \E \Big(  | R(X)|^2\Big) \ll \frac{1}{2^{\ell_0}} \sum_{p |r \Rightarrow p \in I_0} \frac{4^{\Omega(r)} d(r)^4}{r}   \ll \frac{ (\log q)^{O(1)}}{2^{\ell_0}} \ll \exp\big( -C (\log \log q)^{15/4}\big),
 \end{align}
 for some $C>0$.

 Now we bound $\E \Big( \big| \prod_{p \in I_0} L_{p,j_1,j_2,j_3,j_4}( 1/2+s,X )\big|^2\Big) $. We rewrite
 \begin{align*}
    \E \Bigg( & \bigg| \prod_{p \in I_0} L_{p,j_1,j_2,j_3,j_4} ( \tfrac{1}{2}+s,X )\bigg|^2\Bigg) = \E \Bigg(  \sum_{p|m_1n_1 \Rightarrow p \in I_0} \frac{ (\chi_{j_1} * \chi_{j_2})(m_1) (\overline{\chi_{j_3}} * \overline{\chi_{j_4}})(n_1)}{(m_1n_1)^{1/2+s}} X(m_1) \overline{X(n_1)} \\
    & \qquad\qquad\times \sum_{p|m_2n_2 \Rightarrow p \in I_0} \frac{ (\overline{\chi_{j_1}} * \overline{\chi_{j_2}})(m_2) (\chi_{j_3} * \chi_{j_4})(n_2)}{(m_2n_2)^{1/2+\overline{s}}} \overline{X(m_2)} X(n_2)\Bigg)  \\
    & = \prod_{p \in I_0} \bigg(  \sum_{\substack{a_1,a_2,b_1,b_2 \geq 0 \\ a_1+b_2=a_2+b_1}} \frac{(\chi_{j_1} * \chi_{j_2})(p^{a_1}) (\overline{\chi_{j_3}} * \overline{\chi_{j_4}})(p^{b_1}) (\overline{\chi_{j_1}} * \overline{\chi_{j_2}})(p^{a_2})(\chi_{j_3} * \chi_{j_4})(p^{b_2})}{p^{(1/2+s)(a_1+b_1)+(1/2+\overline{s})(a_2+b_2)}}\bigg) \\
    &= \prod_{p \in I_0} \bigg(1+ O \Big( \frac{1}{p^{1+2\Re(s)}} \Big) \bigg).
 \end{align*}
 Since $\Re(s) \geq - (\log \log q)^2/\log q,$ we have that 
 $$p^{-2\Re(s)} \leq p^{2(\log \log q)^2/\log q} \leq \exp \big( 2 (\log \log q)^{-3}\big) \ll 1.$$
 Hence 
 \begin{equation}
 \label{smallprimeslog}
 \prod_{p \in I_0} \bigg(1+ O \Big( \frac{1}{p^{1+2\Re(s)}} \Big) \bigg) \ll \prod_{p \in I_0} \bigg(1+ O \Big( \frac{1}{p}  \Big) \bigg) \ll (\log q)^{O(1)}.
 \end{equation}
 Combining the equation above and \eqref{cs} and \eqref{exp_r} finishes the proof. 
\end{proof}

From the above lemmas we obtain the following corollary.

\begin{corollary}
\label{fs_bound}
    For $u_j \ll 1$, with $j=1,\ldots,4,$ and $- \frac{(\log \log q)^2}{\log q} \ll \Re(s) \leq 2, |\Im(s)| \leq q$, we have
    $$|F(s;\textbf{u})| \ll (\log q)^{O(1)}.$$
\end{corollary}
\begin{proof}
    The proof easily follows by combining Lemmas \ref{bigprimes}, \ref{medium}, and \ref{no_restriction} (see \eqref{smallprimeslog}).
\end{proof}

We will now evaluate the contribution from the small primes. 
\begin{lemma}
\label{primebyprime}
Suppose that $p^{-\Re(s)} \leq 2$. Then 
\begin{align*}  \E  \Big(  L_{p,j_1,j_2,j_3,j_4} ( \tfrac{1}{2}+s,X ) \widetilde{M}_p(X) \Big) &= 1+  \frac{1}{p}\Big(\frac{(\chi_{j_1}*\chi_{j_2})(p)}{p^{s}}-(\chi_1*\chi_2)(p)\,a(p;K)\Big) \\
&\quad \times \Big(\frac{(\overline{\chi_{j_3}}*\overline{\chi_{j_4}})(p)}{p^{s}}  -(\overline{\chi_3}*\overline{\chi_4})(p)\,a(p;K)\Big) + O \Big(  \frac{1}{p^2} \Big),\\
 \E \Big(  L_{p,j_1,j_2,j_3,j_4} ( \tfrac{1}{2}+s,X ) \widetilde{M}_p(X) X(p) \Big) &= \frac{1}{\sqrt{p}} \Big(\frac{(\overline{\chi_{j_3}}*\overline{\chi_{j_4}})(p)}{ p^{s}} - (\overline{\chi_3}*\overline{\chi_4})(p) a(p;K) \Big)\\
 &\quad+ O \Big(  \frac{1}{p^{3/2}}\Big),\\
 \E \Big(  L_{p,j_1,j_2,j_3,j_4} ( \tfrac{1}{2}+s,X) \widetilde{M}_p(X) \overline{X(p)} \Big) &=\frac{1}{\sqrt{p}} \Big(\frac{(\chi_{j_1}*\chi_{j_2})(p) }{p^{s}} - (\chi_1*\chi_2)(p) a(p;K)\Big)\\
 &\quad+ O \Big(  \frac{1}{p^{3/2}}\Big).
\end{align*}
Moreover, for $k \in \mathbb{Z}$, we have
$$ \E \Big( L_{p,j_1,j_2,j_3,j_4} ( \tfrac{1}{2}+s,X ) \widetilde{M}_p(X) X(p)^k \Big) \ll \frac{1}{p^{|k|/2}}.  $$
\end{lemma}
\begin{proof}
We have that
   \kommentar{
    \begin{align*}
        \E &  \Big(  L_{p,j_1,j_2,j_3,j_4} \big( \tfrac{1}{2}+s,X \big) \widetilde{M}_p(X) \Big) = \E \Bigg( \Big( \sum_{a,b \geq 0} \frac{(\chi_{j_1}*\chi_{j_2})(p^{a})(\overline{\chi_{j_3}}*\overline{\chi_{j_4}})(p^{b})}{p^{(a+b)(1/2+s)}}X(p^{a}) \overline{X(p^{b})} \Big) \\
        & \times \Big(1 - \frac{X(p) \chi_1(p) a(p;K)}{\sqrt{p}} \Big)\Big(1 - \frac{X(p) \chi_2(p) a(p;K)}{\sqrt{p}} \Big) \Big(1 - \frac{\overline{X(p) \chi_3(p)} a(p;K)}{\sqrt{p}} \Big) \\
& \times \Big(1 - \frac{ \overline{X(p) \chi_4(p)}a(p;K)}{\sqrt{p}} \Big) \Bigg) = 1+ \frac{ (\chi_1 * \chi_2)(p) (\overline{\chi_3}*\overline{\chi_4})(p)}{p^{1+2s}}- \frac{2a(p;K)(\chi_1*\chi_2)(p) (\overline{\chi_3}*\overline{\chi_4})(p)}{p^{1+s}} \\
&+ \frac{ (\chi_1 * \chi_2)(p) ( \overline{\chi_3}*\overline{\chi_4})(p)a(p;K)^2}{p} +O \Big(  \frac{1}{p^2} \Big) \\
&= 1+ \frac{ (\chi_1 * \chi_2)(p) (\overline{\chi_3}*\overline{\chi_4})(p)(a(p;K)-p^{-s})^2}{p}+ O \Big(  \frac{1}{p^2} \Big). 
    \end{align*}}
 \begin{align*}
&\E   \Big(  L_{p,j_1,j_2,j_3,j_4} \big( \tfrac{1}{2}+s,X \big) \widetilde{M}_p(X) \Big) = \E \Bigg( \Big( \sum_{a,b \geq 0} \frac{(\chi_{j_1}*\chi_{j_2})(p^{a})(\overline{\chi_{j_3}}*\overline{\chi_{j_4}})(p^{b})}{p^{(a+b)(1/2+s)}}X(p^{a}) \overline{X(p^{b})} \Big) \\
        &\qquad\qquad \times \Big(1 - \frac{X(p) \chi_1(p) a(p;K)}{\sqrt{p}} \Big)\Big(1 - \frac{X(p) \chi_2(p) a(p;K)}{\sqrt{p}} \Big)  \\
&\qquad\qquad \times \Big(1 - \frac{\overline{X(p) \chi_3(p)} a(p;K)}{\sqrt{p}} \Big)\Big(1 - \frac{ \overline{X(p) \chi_4(p)}a(p;K)}{\sqrt{p}} \Big) \Bigg)\\
&\ = 1+  \frac{1}{p}\left(\frac{(\chi_{j_1}*\chi_{j_2})(p)}{p^{s}}-(\chi_1*\chi_2)(p)\,a(p;K)\right)\Big(\frac{(\overline{\chi_{j_3}}*\overline{\chi_{j_4}})(p)}{p^{s}}  -(\overline{\chi_3}*\overline{\chi_4})(p)\,a(p;K)\Big)\\
&\qquad\ + O \Big(  \frac{1}{p^2} \Big).
  \end{align*}
  Note that in the case when $(j_1,\ldots,j_4)=(1,\ldots,4)$, the above simplifies to
  \begin{align*}
  \E &  \Big(  L_{p,1,2,3,4} \big( \tfrac{1}{2}+s,X \big) \widetilde{M}_p(X) \Big) = 1+ \frac{ (\chi_1 * \chi_2)(p) (\overline{\chi_3}*\overline{\chi_4})(p)(a(p;K)-p^{-s})^2}{p}+ O \Big(  \frac{1}{p^2} \Big).
  \end{align*}
  The last two identities and the last bound follow similarly.
\end{proof}

The following lemma expresses the expected value with the exponential in terms of the expected value of the mollified contribution.

\begin{lemma}
\label{smallprime}
   For $\Re(s) \gg - \frac{ (\log \log q)^2}{\log q}$ and $p \in I_0$, we have 
    \begin{align*}
        \E & \Bigg(L_{p,j_1,j_2,j_3,j_4}( \tfrac{1}{2}+s,X) \widetilde{M_p}(X) \exp \bigg(  i \frac{a(p;K)}{
    \sqrt{p}}\sum_{j=1}^4   u_j \Re \big(X(p) \chi_j(p)\big)\bigg) \Bigg) \\
    & =  \E \Big(L_{p,j_1,j_2,j_3,j_4}( \tfrac{1}{2}+s,X) \widetilde{M_p}(X) \Big) +\frac{ia(p;K)}{2p} \bigg( \Big(\frac{(\chi_{j_1}*\chi_{j_2})(p)}{p^{s}}- a(p;K)\,(\chi_1*\chi_2)(p) \Big) \\
    &\qquad\times \sum_{j=1}^4 u_j \overline{\chi_j(p)}+\Big(  \frac{(\overline{\chi_{j_3}}*\overline{\chi_{j_4}})(p)}{p^{s}} - a(p;K)(\overline{\chi_3}*\overline{\chi_4})(p)\Big)\sum_{j=1}^4 u_j \chi_j(p)\bigg)\\
    &\qquad- \frac{a(p;K)^2 |\sum_{j=1}^4 u_j \chi_j(p)|^2}{4p}+ O \Big(  \frac{\sum_{j=1}^4 |u_j|}{p^{2+\Re(s)}}\Big).
    \end{align*}
\end{lemma}
\begin{proof}
    Taking the difference and expanding the exponential, we have that
    \begin{align*}
 \E & \Bigg(L_{p,j_1,j_2,j_3,j_4}( \tfrac{1}{2}+s,X) \widetilde{M_p}(X) \exp \bigg(  i \frac{a(p;K)}{
    \sqrt{p}}\sum_{j=1}^4   u_j \Re \big(X(p) \chi_j(p)\big)\bigg) \Bigg) \\
    & -  \E \Bigg(L_{p,j_1,j_2,j_3,j_4}( \tfrac{1}{2}+s,X) \widetilde{M_p}(X) \Bigg)\\
    &\qquad = \E \Bigg( L_{p,j_1,j_2,j_3,j_4}( \tfrac{1}{2}+s,X) \widetilde{M_p}(X) \sum_{k \geq 1} \frac{i^k a(p;K)^k}{k! p^{k/2}}\bigg( \sum_{j=1}^4 u_j \Re \big(X(p) \chi_j(p)\big) \bigg)^k \Bigg).
    \end{align*}
    The contribution from $k \geq 3$ above will be bounded by $O \big( \frac{\sum_{j=1}^4 |u_j|^3}{p^{2}}\big)$, which follows from Lemma \ref{primebyprime}. Hence we need to consider the contributions from $k=1,2$, which are equal to 
    \begin{align*}
&\E \Bigg( \sum_{a,b \geq 0} \frac{ (\chi_{j_1} * \chi_{j_2})(p^{a}) (\overline{\chi_{j_3}} * \overline{\chi_{j_4}})(p^{b})}{p^{(a+b)(\frac{1}{2}+s)}}  \Big(1 - \frac{X(p) \chi_1(p) a(p;K)}{\sqrt{p}} \Big)\Big(1 - \frac{X(p) \chi_2(p) a(p;K)}{\sqrt{p}} \Big) \\
&\qquad \times \Big(1 - \frac{\overline{X(p) \chi_3(p)} a(p;K)}{\sqrt{p}} \Big)  \Big(1 - \frac{ \overline{X(p) \chi_4(p)} a(p;K)}{\sqrt{p}} \Big)\\
&\qquad\times\sum_{k  \in \{1,2\}} \frac{i^k a(p;K)^k}{k! p^{k/2}} \bigg( \sum_{j=1}^4 u_j \Re \big(X(p) \chi_j(p)\big) \bigg)^k \Bigg).
\end{align*}
Using Lemma \ref{primebyprime}, the contribution from $k=1$ is equal to 
\begin{equation}
\label{k1}\frac{ia(p;K)}{2p} \big(\overline{P}F_{j_1,j_2,j_3,j_4;1}(s)+P F_{j_1,j_2,j_3,j_4;2}(s)\big)+ O \Big(  \frac{\sum_{j=1}^4 |u_j|}{p^{2+\Re(s)}}\Big),
\end{equation}
where, for simplicity of notation,
\begin{align*}
    P& = \sum_{j=1}^4 u_j \chi_j(p),\\
F_{j_1,j_2,j_3,j_4;1}(s)&= \frac{(\chi_{j_1}*\chi_{j_2})(p)}{p^{s}} - a(p;K)\,(\chi_1*\chi_2)(p),\\
 F_{j_1,j_2,j_3,j_4;2}(s)&=  \frac{(\overline{\chi_{j_3}}*\overline{\chi_{j_4}})(p)}{p^{s}} - a(p;K)(\overline{\chi_3}*\overline{\chi_4})(p).
 \end{align*}
The contribution for $k=2$, again by Lemma \ref{primebyprime}, is equal to
\begin{equation} 
\label{k2}
- \frac{a(p;K)^2 |P|^2}{4p} + O\Big(  \frac{ \sum_{j=1}^4 |u_j|^2}{p^{2+2\Re(s)}}\Big).
\end{equation}
 The conclusion follows by combining equations \eqref{k1} and \eqref{k2}, together with the bound for the contribution from $k \geq 3$.    
  \end{proof}

  We now focus on the diagonal case $(j_1,\ldots,j_4)=(1,\ldots,4)$ and we will prove the following lemma.
  \begin{lemma}
  \label{1234}
   We have
   \begin{align*}
       \E & \Bigg( \prod_{p \in I_0} L_{p,1,2,3,4}( \tfrac{1}{2},X) \widetilde{M_p}(X) \exp \bigg(  i \frac{a(p;K)}{
    \sqrt{p}}\sum_{j=1}^4   u_j \Re \big(X(p) \chi_j(p)\big)\bigg) \Bigg) \\
    & = \E \Bigg( \prod_{p \in I_0} L_{p,1,2,3,4}( \tfrac{1}{2},X) \widetilde{M_p}(X) \Bigg) \exp \Big(  - \frac{\sum_{j=1}^4 u_j^2}{4} \log \log q_0 \Big) \\
    &\qquad \times \prod_{j=1}^4 \prod_{p|D_j} \bigg(  1+ \frac{u_j^2}{4p} + O \Big( \frac{u_j^4}{p^2} \Big)\bigg)\exp \bigg( O \Big(  \log \log \log q \sum_{j=1}^4 |u_j|^2\Big) \bigg)  \bigg(1 +O \Big(  \sum_{j=1}^4 |u_j|\Big) \bigg). 
   \end{align*} 
  \end{lemma}
  \begin{proof}
   From Lemma \ref{smallprime} and using the fact that
   $$  \prod_{p \in I_0} L_{p,1,2,3,4}\big( \tfrac{1}{2},X) \widetilde{M_p}(X) \neq 0,$$
 which follows from Theorem $1.5$ in \cite{BFM}, we have  
\begin{align}\label{intermediary_step}
&\E \Bigg( \prod_{p \in I_0} L_{p,1,2,3,4}( \tfrac{1}{2},X) \widetilde{M_p}(X) \exp \bigg(  i \frac{a(p;K)}{
    \sqrt{p}}\sum_{j=1}^4   u_j \Re \big(X(p) \chi_j(p)\big)\bigg) \Bigg)\nonumber\\
    &\quad= \E \Bigg( \prod_{p \in I_0} L_{p,1,2,3,4}( \tfrac{1}{2},X) \widetilde{M_p}(X) \Bigg)  \\
    &\qquad \times  \prod_{p \in I_0} \Bigg( 1 + \frac{ \frac{ i a(p;K)}{2p} \big( \overline{P} F_{1,2,3,4;1}(0)+ P F_{1,2,3,4;2}(0) \big) - \frac{a(p;K)^2 |P|^2}{4p} +  O \big(  \frac{\sum_{j=1}^4 |u_j|}{p^{2}}\big)}{  \E \Big( L_{p,1,2,3,4}( \frac{1}{2},X) \widetilde{M_p}(X)\Big) } \Bigg).\nonumber  
    \end{align}
 Note that 
   \begin{equation}
       a(p;K) = 1 + O \Big( \frac{ \log p}{\log q} \Big), \label{apk}
        \end{equation} and
   $$F_{1,2,3,4;1}(0) = (\chi_1*\chi_2)(p) \big(1-a(p;K)\big),$$
   $$F_{1,2,3,4;2}(0) = (\overline{\chi_3}*\overline{\chi_4})(p) \big(1-a(p;K)\big).$$
   From \eqref{apk} we get that for $j=1,2$,
   \begin{equation}
   \label{fi} F_{1,2,3,4;j}(0) \ll \frac{ \log p}{\log q}.
   \end{equation}
From Lemma \ref{primebyprime}, we also have that 
\begin{align*}
    \E \Big( L_{p,1,2,3,4}( \tfrac{1}{2},X) \widetilde{M_p}(X) \Big) = 1+ \frac{F_{1,2,3,4;1}(0) F_{1,2,3,4;2}(0)}{p}+ O \Big( \frac{1}{p^2}\Big) = 1+ O \Big( \frac{1}{p^2} + \frac{ (\log p)^2}{p (\log q)^2} \Big),
\end{align*}
where the second identity follows from \eqref{fi}. 
From \eqref{intermediary_step} and the equation above, it follows that 
\begin{align*}
&\E \Bigg( \prod_{p \in I_0} L_{p,1,2,3,4}( \tfrac{1}{2},X) \widetilde{M_p}(X) \exp \bigg(  i \frac{a(p;K)}{
    \sqrt{p}}\sum_{j=1}^4   u_j \Re \big(X(p) \chi_j(p)\big)\bigg) \Bigg)\\ 
    &\qquad = \E \Bigg( \prod_{p \in I_0} L_{p,1,2,3,4}( \tfrac{1}{2},X) \widetilde{M_p}(X) \Bigg) \prod_{p \in I_0} \bigg(  1 - \frac{|P|^2}{4p} + O \Big(\frac{\sum_{j=1}^4 |u_j|}{p^2}+  \frac{ \log p\sum_{j=1}^4 |u_j| }{p \log q} \Big)\bigg)\\
    &\qquad =   \E \Bigg( \prod_{p \in I_0} L_{p,1,2,3,4}( \tfrac{1}{2},X) \widetilde{M_p}(X) \Bigg) \exp \Big( - \sum_{p \in I_0} \frac{|P|^2}{4p} \Big) \bigg( 1 + O \Big( \sum_{j=1}^4 |u_j| \Big)\bigg). 
    \end{align*}
     Note that we have
   \begin{align*}
       \sum_{p \in I_0 } \frac{|P|^2}{4p} & = \sum_{\substack{j,k=1}}^4 \sum_{p \in I_0} \frac{ u_j \overline{u_k} \chi_j(p) \overline{\chi_k(p)}}{4p} \\
       &= \sum_{j=1}^4 \sum_{p \in I_0} \frac{u_j^2}{4p} - \sum_{j=1}^4 \sum_{p |D_j} \frac{u_j^2}{4p}+ \sum_{\substack{j,k=1\\j \neq k}}^4 \sum_{p \in I_0} \frac{ u_j \overline{u_k} \chi_j(p) \overline{\chi_k(p)}}{4p}.
   \end{align*}
 Using equation $(7.28)$ in \cite{BFM}, it follows that  
   $$ \sum_{\substack{j,k=1\\j \neq k}}^4  \sum_{p \in I_0} \frac{ u_j \overline{u_k} \chi_j(p) \overline{\chi_k(p)}}{4p} \ll \log\log\log q \sum_{j=1}^4 |u_j|^2 .$$
   Combining the three equations above completes the proof of the lemma.
     \end{proof}

  Now assume that $(j_1,\ldots,j_4)$ corresponds to either one of the four one-swap terms or the two-swap term (as in equation \eqref{lx}). The following lemma follows similarly to Lemma \ref{1234}, so we will be brief with the proof.

\begin{lemma}\label{not1234}
     We have 
   \begin{align*}
       \E & \Bigg( \prod_{p \in I_0} L_{p,j_1,j_2,j_3,j_4}( \tfrac{1}{2},X) \widetilde{M_p}(X) \exp \bigg(  i \frac{a(p;K)}{
    \sqrt{p}}\sum_{j=1}^4   u_j \Re \big(X(p) \chi_j(p)\big)\bigg) \Bigg) \\
    & \ll  \E \Bigg( \prod_{p \in I_0} L_{p,j_1,j_2,j_3,j_4}( \tfrac{1}{2},X) \widetilde{M_p}(X) \Bigg) \exp \Big(  - \frac{\sum_{j=1}^4 u_j^2}{4} \log \log q_0 \Big) \\
    &\qquad \times \prod_{j=1}^4 \prod_{p|D_j} \bigg(  1+ \frac{u_j^2}{4p} + O \Big( \frac{u_j^4}{p^2} \Big)\bigg)\exp \bigg( O \Big(  \log \log \log q \sum_{j=1}^4 |u_j|\Big) \bigg) \bigg(1 +O \Big(  \sum_{j=1}^4 |u_j|\Big) \bigg). 
    \end{align*}
    \end{lemma}
    \begin{proof}
    Similarly to the proof of Proposition $7.1$ in \cite{BFM}, for $p$ large enough, say $p \geq c_0$ for some $c_0>0$, we have that $\E \Big( L_{p,j_1,j_2,j_3,j_4}( \frac{1}{2},X) \widetilde{M_p}(X) \Big) \neq 0$. 
Then we have that
\begin{align} 
\label{ll}
  &\E \Bigg( \prod_{p \in I_0} L_{p,j_1,j_2,j_3,j_4}( \tfrac{1}{2},X) \widetilde{M_p}(X) \exp \bigg(  i \frac{a(p;K)}{
    \sqrt{p}}\sum_{j=1}^4   u_j \Re \big(X(p) \chi_j(p)\big)\bigg) \Bigg) \nonumber\\
    &\qquad\ll   \E  \Bigg( \prod_{c_0<p\leq q_0} L_{p,j_1,j_2,j_3,j_4}( \tfrac{1}{2},X) \widetilde{M_p}(X) \exp \bigg(  i \frac{a(p;K)}{
    \sqrt{p}}\sum_{j=1}^4   u_j \Re \big(X(p) \chi_j(p)\big)\bigg) \Bigg), 
    \end{align}
and now we can proceed as in the proof of Lemma \ref{1234}. Namely, using that in the non-diagonal case,
$$ \E  \Big(L_{p,j_1,j_2,j_3,j_4}( \tfrac{1}{2},X) \widetilde{M_p}(X)  \Big) =1 + O \Big( \frac{1}{p} \Big),  $$ we obtain that
\begin{align}\label{interm3}
      \E & \Bigg( \prod_{c_0<p \leq q_0} L_{p,j_1,j_2,j_3,j_4}( \tfrac{1}{2},X) \widetilde{M_p}(X) \exp \bigg(  i \frac{a(p;K)}{
    \sqrt{p}}\sum_{j=1}^4   u_j \Re \big(X(p) \chi_j(p)\big)\bigg) \Bigg) \\
    &=   \E  \Bigg( \prod_{c_0<p \leq q_0} L_{p,j_1,j_2,j_3,j_4}( \tfrac{1}{2},X) \widetilde{M_p}(X)  \Bigg)\prod_{c_0<p \leq q_0} \bigg(1 + O \Big( \frac{1}{p} \Big) \bigg)\nonumber\\
    &\qquad\times \bigg( \frac{ i a(p;K)}{2p} \big( \overline{P} F_{j_1,j_2,j_3,j_4;1}(0)+ P F_{j_1,j_2,j_3,j_4;2}(0) \big)- \frac{a(p;K)^2 |P|^2}{4p} +  O \Big(  \frac{\sum_{j=1}^4 |u_j|}{p^{2}}\Big)\bigg) \Bigg).\nonumber
\end{align}
        
   Note that we have 
   $$F_{j_1,j_2,j_3,j_4;1}(0) = (\chi_{j_1}*\chi_{j_2})(p)-(\chi_1*\chi_2)(p) + O \Big(  \frac{\log p}{\log q}\Big), $$
   $$F_{j_1,j_2,j_3,j_4;2}(0) = (\overline{\chi_{j_3}}*\overline{\chi_{j_4}})(p)-(\overline{\chi_3}*\overline{\chi_4})(p) + O \Big(  \frac{\log p}{\log q}\Big).$$
 By expanding the term involving $ia(p;K)/2p$, we rewrite 
 \begin{align*}
     \eqref{interm3} &= \E  \Bigg( \prod_{c_0<p \leq q_0} L_{p,j_1,j_2,j_3,j_4}( \tfrac{1}{2},X) \widetilde{M_p}(X)  \Bigg) \\
     &\qquad\times\prod_{c_0<p \leq q_0} \bigg( 1- \frac{|P|^2}{4p}+ \frac{i}{2p} \sum_{\psi} b(\psi) \psi(p)+ O \Big(  \frac{ \sum_{j=1}^4 |u_j|}{p^2}\Big) \bigg), 
 \end{align*}
   where the sum above is over $\psi$ of the form $\chi_j \overline{\chi_k}$ (with $j\neq k$) and the coefficients $b(\psi) \in \{u_1,\ldots,u_4\}.$
   Similarly to before, since
   $$ \sum_{p \in I_0} \frac{\psi(p)}{p} \ll \log \log \log q,$$ we get that 
   \begin{align*}
       \eqref{interm3}& = \E  \Bigg( \prod_{c_0<p \leq q_0} L_{p,j_1,j_2,j_3,j_4}( \tfrac{1}{2},X) \widetilde{M_p}(X)  \Bigg) \exp \Big( -\sum_{c_0 < p \leq q_0} \frac{|u_j|^2}{4p} \Big) \\
       &\quad \times \prod_{j=1}^4 \prod_{\substack{p|D_j \\ p >c_0}} \bigg(1+ \frac{u_j^2}{4p}+ O \Big( \frac{u_j^4}{p^4}\Big) \bigg)\exp \bigg( O \Big(  \log \log \log q \sum_{j=1}^4 |u_j|\Big) \bigg) \bigg(1 +O \Big(  \sum_{j=1}^4 |u_j|\Big) \bigg).
   \end{align*}
   Using the equation above and \eqref{ll} finishes the proof.
   \end{proof}

\section{Proof of Proposition \ref{main_prop} and Proposition \ref{mainprop}}  
\label{proofs}
\subsection{Proof of Proposition \ref{main_prop}}\label{proofmainprop}

We now return to the proof of Proposition \ref{main_prop}.

We shift the contour of integration in \eqref{mellin} to $\Re(s)= - \frac{ (\log \log q)^2}{\log q}$, encountering a pole at $s=0$. After truncating the integral to $|\Im(s)| \leq A \log q$, for some large $A>0$, and then using Corollary \ref{fs_bound} to bound the vertical and horizontal contours, we get
    \begin{align}
         &\E\Bigg( L_V(X) M(X) \exp\bigg(  i\sum_{j=1}^4 u_jP_{\chi_j}(X)\bigg)\Bigg) \nonumber \\
         &\qquad= \E \Bigg(L(\tfrac12,X)M(X) \exp\bigg(  i\sum_{j=1}^4 u_jP_{\chi_j}(X)\bigg) \Bigg) + O\big( (\log q)^{-2000}\big).\label{relate_exp}
    \end{align}
    
    From Proposition $7.2$ in \cite{BFM} (see the equation following equation $(7.30)$), note that for $(j_1,\ldots,j_4)$ not equal to the diagonal term (i.e., $j_k=k$ for all $k$), 
\begin{equation*}
    \E \Big( L_{j_1,j_2,j_3,j_4} ( \tfrac{1}{2},X  ) M(X)  \Big) \ll (\log q)^{-3}.
\end{equation*}Using the expression \eqref{lx} for $L(s,X)$, equation \eqref{relate_exp} and Lemmas \ref{bigprimes}, \ref{no_restriction}, \ref{1234}, and \ref{not1234}, and using once more Proposition $7.2$ in \cite{BFM}, we get that
\begin{align*}
    &\E  \Bigg(L(\tfrac12,X)M(X) \exp\bigg(  i\sum_{j=1}^4 u_jP_{\chi_j}(X)\bigg) \Bigg)\\
    &\qquad= \bigg( \E \Big(L\big(\tfrac12,X\big)M(X)\Big) + O\big(  (\log q)^{-3}\big) \bigg)\exp \Big(  - \frac{\sum_{j=1}^4 u_j^2}{4} \log \log q_0 \Big) \\
    &\qquad\qquad \times  \prod_{j=1}^4 \prod_{p|D_j} \bigg(  1+ \frac{u_j^2}{4p} + O \Big( \frac{u_j^4}{p^2} \Big)\bigg) \exp \bigg( O \Big(  \log \log \log q \sum_{j=1}^4 u_j^2\Big) \bigg)\bigg(1 +O \Big(  \sum_{j=1}^4 |u_j|\Big) \bigg).
\end{align*}
Now note that 
\begin{align*}
    \prod_{p|D_j} \bigg(  1+ \frac{u_j^2}{4p} + O \Big( \frac{u_j^4}{p^2} \Big)\bigg) \ll \exp \Big( \frac{u_j^2}{4} \sum_{p|D_j} \frac{1}{p} \Big).
\end{align*}
Using the fact that
$$ \sum_{p|D_j} \frac{1}{p} \ll \log \log \log D_j \ll \log \log \log q,$$ it follows that 
\begin{align*}
 \prod_{j=1}^4    \prod_{p|D_j} \bigg(  1+ \frac{u_j^2}{4p} + O \Big( \frac{u_j^4}{p^2} \Big)\bigg) \ll \exp \bigg( O \Big( \log \log \log q \sum_{j=1}^4 u_j^2\Big) \bigg).
 \end{align*}
This finishes the proof of Proposition \ref{main_prop}.

\subsection{Proof of Proposition \ref{mainprop}}\label{sectionmainprop}
Using Corollary \ref{fchirandom}, we have that 
\begin{align*}
   & \frac{1}{\varphi^{*}(q)}\ \sumstar_{\chi \pmod q} F(\chi)  
\exp\bigg(  i\sum_{j=1}^4 u_jP(\chi\chi_j)\bigg) \\
&\qquad= \E \Bigg( F(X) \exp\bigg(  i\sum_{j=1}^4 u_jP_{\chi_j}(X)\bigg)\Bigg) +O\big(  (\log q)^{-2000}\big).
\end{align*}
    Applying Proposition \ref{main_prop}, we find that the expected value above equals
    \begin{align*}
    &\E\Big(L\big(\tfrac12,X\big)M(X)\Big)  \exp \Big(  - \frac{\sum_{j=1}^4 u_j^2}{4} \log \log q_0 \Big) \\
    &\qquad \times \exp \bigg( O \Big(  \log \log \log q \sum_{j=1}^4 u_j^2\Big) \bigg) \bigg(1 +O \Big(  \sum_{j=1}^4 |u_j|\Big) \bigg) +O\big(  (\log q)^{-3}\big).
    \end{align*}
    Taking $u_j=0$ in \eqref{relate_exp} and appealing to Corollary \ref{fchirandom} once more, we obtain
    \begin{align*}
 &   \frac{1}{\varphi^{*}(q)}\   \sumstar_{\chi \pmod q}  F(\chi)  
\exp\bigg(  i\sum_{j=1}^4 u_jP(\chi\chi_j)\bigg)  = \bigg( \frac{1}{\varphi^{*}(q)}\ \sumstar_{\chi \pmod q} F(\chi)  + O \big( (\log q)^{-2000} \big) \bigg)    \\
&\qquad \times  \exp \Big(  - \frac{\sum_{j=1}^4 u_j^2}{4} \log \log q_0 \Big) \exp \bigg( O \Big(  \log \log \log q \sum_{j=1}^4 u_j^2\Big) \bigg) \bigg(1 +O \Big(  \sum_{j=1}^4 |u_j|\Big) \bigg) \\
&\qquad+ O \big( (\log q)^{-3}\big).
    \end{align*}
The conclusion follows upon using \eqref{weight}.

\section{Proof of Theorem \ref{thm:cltjoint} using Proposition \ref{mainprop}}  \label{sec:proofcltjoint}

Let
\[
\Phi_q(\textbf{u})=\frac{1}{\varphi_F^{*}(q)}\ \sumstar_{\chi \pmod q} F(\chi)  
e\bigg( - \frac{\sum_{j=1}^4 u_jP(\chi\chi_j)}{\sqrt{\frac12 \log \log q}}\bigg),
\]
where, as usual,  $e(x)=e^{2\pi i x}$.
From Proposition \ref{mainprop}, by making the change of variables 
$$ u_j \mapsto - \frac{2 \pi u_j}{ \sqrt{ \frac{1}{2} \log \log q}}$$ and noting that $\log\log q_0=\log\log q+O(\log\log\log q)$, we obtain that
\begin{align}\label{thm:mainresult2}
\Phi_q(\textbf{u})& = e^{-2 \pi^2 \sum_{j=1}^4u_j^2}   \\
&\qquad \times   \exp \bigg( O \Big( \frac{ \log \log \log q }{\log \log q} \sum_{j=1}^4 u_j^2\Big) \bigg) \bigg(1 +O \Big(  \frac{1}{\sqrt{\log \log q}}\sum_{j=1}^4  |u_j|\Big) \bigg) + O \Big(  (\log q)^{-3}\Big),\nonumber
\end{align}
for $u_j\ll\sqrt{\log\log q}$. 

This is the key input into the proof of Theorem \ref{thm:cltjoint}. We also require some additional results.

Let $1\leq j\leq 4$. For $u \in \mathbb{R}$, set
\[
\Psi_q(u):=\frac{1}{\varphi^*(q)} \, \sumstar_{\chi \pmod  q}
e \bigg(-\frac{uP(\chi\chi_j)}{\sqrt{\frac12\log\log q}}\bigg). 
\]
We will prove the following, which is a simpler version of Proposition \ref{mainprop} and \eqref{thm:mainresult2}.

\begin{lemma}\label{lem:cf}
For $u\in\mathbb R$ with $u\ll \sqrt{\log\log q}$ we have
\[
\Psi_q(u)=e^{-2\pi^2u^2}\exp \bigg( O \Big( \frac{ u^2\log \log \log q }{\log \log q} \Big) \bigg)
+O\bigl((\log q)^{-2000}\bigr).
\]
\end{lemma}

\begin{proof}
By repeating the argument in Lemma \ref{random_chi} (with $F(\chi)$ being replaced by $1$) we get that 
\[
\frac1{\varphi^*(q)}\ \sumstar_{\chi \pmod q}\exp \bigl(iu P(\chi\chi_j)\bigr)
=\mathbb E \Big(\exp \bigl(iu P_{\chi_j}(X)\bigr)\Big)+O\bigl((\log q)^{-2000}\bigr),
\]
for $u\ll 1$.
Taylor expanding and using the
independence of the $X(p)$, we get that 
\begin{align*}
\mathbb E \Big(\exp \bigl(iu P_{\chi_j}(X)\bigr) \Big)
&=\prod_{p\in I_0}\mathbb E \Bigg(\exp \bigg(iu\,\Re\Big(\frac{X(p)\chi_j(p)a(p;K)}{\sqrt p}\Big)\bigg)\Bigg)\\
&=\prod_{\substack{p\in I_0 \\ p \nmid D_j}}\bigg(1-\frac{u^2a(p;K)^2}{4p}+O\Big(\frac{u^4}{p^2}\Big)\bigg).
\end{align*}
As in the proof of Proposition \ref{mainprop}, we get that 
\[
\mathbb E \Big(\exp \bigl(iu P_{\chi_j}(X)\bigr)\Big)=\exp\Big(-\frac{u^2}{4}\log\log q_0\Big)\exp \Big( O ( u^2  \log \log \log q )\Big).
\]
Making the change of variables
$u\mapsto -\frac{2\pi u}{\sqrt{\frac12\log\log q}}$ finishes the proof of the lemma.
\end{proof}

\begin{lemma} \label{lem:cheby} Assume GRH. Let $1\leq j\leq 4$.
Let $\Delta \ge 1$ and let $B>0$ be sufficiently large.
Then for all primitive characters $\chi$ modulo $q$ outside an exceptional set of size $\ll \frac{q}{\Delta^2}+\frac{q}{(\log \log q)^{2000}}$ the following statements hold:
\begin{enumerate}
    \item  $ \displaystyle |L(\tfrac12,\chi\chi_j)M(\tfrac12,\chi\chi_j)| \le \Delta$;
    \item $\displaystyle \prod_{k=1}^K |M_{k}(\tfrac12,\chi\chi_j)| \le (\log \log q)^B$;
    \item $ \displaystyle \frac{1}{\sqrt{\frac12\log \log q}}\bigg| \sum_{p\in I_0} \frac{ \chi\chi_j(p)a(p;K)}{ \sqrt{p}} \bigg|\le \log \log \log q $.
\end{enumerate}
\end{lemma}
\begin{proof}

By Chebyshev's inequality and \eqref{upperbd},
\[
\sumstar_{\substack{\chi \pmod q \\ |L(\frac12,\chi\chi_j)M(\frac12,\chi\chi_j)| > \Delta} } 1  \le \frac{1}{\Delta^2}\ \sumstar_{\chi \pmod q} |L(\tfrac12,\chi\chi_j)M(\tfrac12,\chi\chi_j)|^2 \ll \frac{q}{\Delta^2}.
\]

Similarly,
\begin{align*}
\sumstar_{\substack{\chi \pmod q \\  \prod_{k=1}^K |M_{k}(\frac12,\chi\chi_j)| > (\log \log q)^B} } 1 &\le \frac{1}{(\log \log q)^{2B}} \ \sumstar_{\chi \pmod q} \prod_{k=1}^K |M_{k}(\tfrac12,\chi\chi_j)|^2  \\
&\ll   \frac{q}{(\log \log q)^{2B}} \prod_{q_0< p \le q_K}\bigg( 1+O\Big( \frac{1}{p}\Big)\bigg) \ll \frac{q}{(\log \log q)^{2000}}.
\end{align*}

Finally, we note that the arguments given in the proofs of Lemma \ref{momentsP} and Lemma \ref{lem:largedev} show that the conclusion of Lemma \ref{lem:largedev} holds with $P(\chi\chi_j)$ being replaced by $\sum_{p\in I_0} \frac{\chi\chi_j(p)a(p;K)}{ \sqrt{p}}$, so using this result with $V=\log \log \log q$ gives that
\begin{align*}
&\# \bigg\{ \chi \pmod q,\chi\ne\chi_0 : \bigg|\sum_{p\in I_0} \frac{\chi\chi_j(p)a(p;K)}{ \sqrt{p}} \bigg| \ge \log \log \log q \sqrt{\frac12\log \log q} \bigg\}\\
&\qquad \ll \frac{q}{(\log \log q)^{2000}}, 
\end{align*}
which completes the proof of the lemma.
\end{proof}

The following result is due independently to Beurling and Selberg (see \cite{V}).

\begin{lemma}\label{lem:BS}
Let $\Delta>0$ and $U=[a,b] \subset \mathbb R$ be an interval. Then there exists an entire function $W_{U,\Delta}(x)$ such that each of the following holds:
\begin{enumerate}
    \item $\displaystyle 0 \le \mathbf{1}_{U}(x)- W_{U,\Delta}(x) \le \left( \frac{\sin(\pi \Delta(x-a)) }{\pi \Delta (x-a)} \right)^2+\left( \frac{\sin(\pi \Delta(b-x)) }{\pi \Delta (b-x)} \right)^2$, $\forall x \in \mathbb R$;
    \item $ \displaystyle \widehat {W_{U,\Delta}}(\xi) =\begin{cases}
    \widehat {\mathbf{1}_{U}}(\xi)+O\left(\frac{1}{\Delta} \right) & \text{ if } |\xi| < \Delta, \xi \in \mathbb R, \\
    0 & \text{ if } |\xi| \ge \Delta, \xi \in \mathbb R.
    \end{cases}$
\end{enumerate}
\end{lemma}

Let $\mathcal{S}_j$ denote the set of primitive characters $\chi$ modulo $q$ satisfying the properties in Lemma \ref{lem:cheby},  
and also $(\chi \chi_j)^2 \neq \chi_0$  and
\[
|L(\tfrac12,\chi\chi_j)M(\tfrac12,\chi\chi_j)|\ge \frac{1}{\Delta},
\]
with 
$$\Delta= \frac{ \sqrt{\log \log q}}{\log\log \log q}.$$ Then for $\chi \in \mathcal{S}_j$ we have that
\begin{align}
\log|L(\tfrac12,\chi\chi_j)|=&- \log |M(\tfrac12,\chi\chi_j)|+O(\log \Delta) \nonumber  \\
=& -\log |M_0(\tfrac12,\chi\chi_j)|+O(\log \Delta). \label{lm}
\end{align}

Recall that we have $$M_0 (  \tfrac{1}{2},\chi \chi_j) = \sum_{\substack{p|n \Rightarrow p \in I_0 \\ \Omega(n) \leq \ell_0}} \frac{ \chi\chi_j(n)\mu(n)   a(n;K)}{\sqrt{n}} . $$
Let $$b(p) = -\frac{\chi \chi_j(p)a(p;K)}{\sqrt{p}} $$
and let $$ e_{\ell}= \sum_{p_1<p_2<\ldots<p_{\ell}\leq q_0} b(p_1) \ldots b(p_{\ell})$$ be the $\ell^{\text{th}}$ elementary symmetric polynomial (in the $b(p_i)$). 
Then
$$ M_0 (  \tfrac{1}{2},\chi \chi_j) = \sum_{\ell \leq \ell_0} e_{\ell}.$$
Looking at the generating series of the elementary symmetric functions, we have that
\begin{align*}
  F(z):=  \sum_{\ell \geq 0} & e_{\ell} z^{\ell} = \prod_{p \in I_0} \big(1+ b(p) z\big) =\exp \Big(  \sum_{p \in I_0} \sum_{k \geq 1} \frac{ (-1)^{k-1}b(p)^k z^k}{k}\Big) \\
    &= \exp \Big(  \sum_{k \geq 1} \frac{ (-1)^{k-1} z^k P_k}{k} \Big),
\end{align*}
where $$P_k = \sum_{p \in I_0} b(p)^k.$$
Note that
$$|P_1| \ll \log \log \log q\sqrt{\log \log q},$$ by condition $(3)$ in Lemma \ref{lem:cheby}. We also have that 
$$| P_2| \ll \sum_{p \in I_0} \frac{1}{p} \ll \log \log q_0. $$
When $k \geq 3$, we have that 
$$|P_k| \leq \sum_{p \in I_0} \frac{1}{p^{k/2}} \leq   \frac{2^{3/2} \zeta(3/2)}{2^{k/2}}.$$
Now we have that
$$M_0 ( \tfrac{1}{2},\chi \chi_j ) = \frac{1}{2 \pi i} \oint_{|z|=r} \frac{F(z)}{(1-z)z^{\ell_0+1}} \, dz,$$
for some $0<r<1$.
In the integral above, we shift the contour of integration to $|z|=R,$ for some $1<R<\sqrt{2},$ and we encounter the pole at $z=1$. The integral on the new contour is bounded by 
$$ \exp \Big(  C_1 \log \log \log q\sqrt{\log \log q} + C_2 \log  \log q + C_3 \sum_{k \geq 3} \frac{R^k}{k2^{k/2}}\Big) R^{-\ell_0},$$ for some constants $C_1,C_2,C_3>0$. Since $R<\sqrt{2}$, the sum over $k \geq 3$ is of size $O(1).$ Since $\ell_0 \asymp (\log \log q)^{15/4}$, it follows that the above error term is bounded by 
$$\exp \big(- C (\log \log q)^{15/4} \big),$$ for some $C>0$. Evaluating the residue at $z=1$ in the integral, we then get that
$$M_0 (  \tfrac{1}{2},\chi \chi_j )= \prod_{p \in I_0} \Big(1 - \frac{\chi \chi_j(p) a(p;K)}{\sqrt{p}} \Big)+O \Big( \exp \big(- C (\log \log q)^{15/4} \big)\Big). $$
Since $\chi \in \mathcal{S}_j$, we again use condition $(3)$ in Lemma \ref{lem:cheby}, and it follows that 
\begin{equation}\label{approxM0}
  M_0 (  \tfrac{1}{2},\chi \chi_j )= \prod_{p \in I_0} \Big(1 - \frac{\chi \chi_j(p) a(p;K)}{\sqrt{p}} \Big)\bigg(1+O \Big( \exp \big(- C (\log \log q)^{15/4} \big)\Big)\bigg). 
\end{equation}
Note that the constant $C$ above can change from line to line. Now we take the logarithm of the equation above, take its real part, and then use the Taylor series for the logarithm. Using the fact that $(\chi \chi_j)^2 \neq \chi_0$, we have
$$ \sum_{p \in I_0} \frac{ (\chi \chi_j)^2(p) a(p;K)^2}{p} \ll \log \log \log q.$$
We then get that
$$ \log |  M_0(  \tfrac{1}{2},\chi \chi_j)|  = - P(\chi \chi_j)+ O ( \log \log \log q ),$$ and hence combining this with \eqref{lm}, we conclude that 
\begin{equation} \label{eq:clean}
P(\chi\chi_j)=\log |L(\tfrac12,\chi\chi_j)|+O(\log\Delta).
\end{equation}
In particular, we have
\begin{equation} \label{eq:lfbound}
\frac{\log |L(\tfrac12,\chi\chi_j)| }{\sqrt{\frac12 \log \log q}}\ll \log \log \log q,
\end{equation}
 for $\chi \in \mathcal {S}_j$.

Denote $$\mathcal L(\chi)= \frac{\log |L(\frac12,\chi)|}{\sqrt{\frac12 \log \log q}}\qquad \text{and}\qquad \mathcal P(\chi)=\frac{P(\chi)}{\sqrt{\frac12\log \log q}}.$$ 
By \eqref{eq:lfbound}, for $\chi \in \mathcal  S_j$ we have that $\mathcal L(\chi\chi_j) \in U_j$ if and only if $\mathcal L(\chi\chi_j) \in \mathcal U_j$ where $\mathcal U_j=U_j \cap [-A \log \log \log q, A \log \log \log q]$ with $A>0$ sufficiently large. We write $\mathcal U_j=[a_j,b_j]$ with $|b_j-a_j| \ll \log \log \log q$.
Using \eqref{eq:clean} we have for  
all $\chi \in \mathcal{S}_j$ that 
\[
\Big|\mathbf{1}_{\mathcal U_j}\big(\mathcal L(\chi\chi_j)\big)-\mathbf{1}_{\mathcal U_j}\big(\mathcal P(\chi\chi_j)\big)\Big|\leq \mathbf{1}_{[a_j-\delta,a_j+\delta]}\big(\mathcal P(\chi\chi_j)\big) +\mathbf{1}_{[b_j-\delta,b_j+\delta]}\big(\mathcal P(\chi\chi_j)\big) \leq 2\mathbf{1}_{U_{\delta,j}}\big(\mathcal P(\chi\chi_j)\big),
\]
where $U_{\delta,j}=[a_j-\delta,a_j+\delta] \cup  [b_j-\delta,b_j+\delta]$ with $\delta=C \frac{\log\Delta}{\sqrt{\log \log q}}$ and $C>0$ sufficiently large. It follows that for $\chi \in \cap_{j=1}^{4} \mathcal{S}_j$, we have
\[
\prod_{j=1}^{4}\mathbf{1}_{U_j}\big(\mathcal L(\chi\chi_j)\big)=\prod_{j=1}^{4}\mathbf{1}_{U_j}\big(\mathcal P(\chi\chi_j)\big)+O\Big(\sum_{j=1}^4\mathbf{1}_{U_{\delta,j}}\big(\mathcal P(\chi\chi_j)\big)\Big).
\]
Using H\"older's inequality, \eqref{upperbd}, and \eqref{weight}, we obtain
\begin{equation} \label{eq:ltop}
\begin{split}
&\frac{1}{\varphi_F^{*}(q)}\ \sumstar_{\chi \in \cap_{j=1}^{4} \mathcal S_j} F(\chi) \, \prod_{j=1}^{4}\mathbf{1}_{U_j}\big(\mathcal L(\chi\chi_j)\big)=\frac{1}{\varphi_F^{*}(q)}\ \sumstar_{\chi \in \cap_{j=1}^{4} \mathcal S_j} F(\chi) \, \prod_{j=1}^{4}\mathbf{1}_{U_j}\big(\mathcal P(\chi\chi_j)\big) 
\\
&\qquad \qquad +O\bigg(q^{1/e_2-1} \bigg(\ \,\sumstar_{\chi \pmod q} \sum_{j=1}^4\mathbf{1}_{U_{\delta,j}}\big(\mathcal P(\chi\chi_j)\big) \bigg)^{1/e_1} \bigg),
\end{split}
\end{equation}
for any $e_1,e_2>1$ such that $1/e_1+1/e_2=1$. We shall choose 
\[
e_1=\frac{1+\big[\frac{1}{\varepsilon}\big]}{\big[\frac{1}{\varepsilon}\big]}=1+O(\varepsilon)\qquad\text{and}\qquad e_2=1+\Big[\frac{1}{\varepsilon}\Big]\asymp \frac{1}{\varepsilon}.
\]

Consider the left-hand side in \eqref{eq:ltop}. By the Cauchy-Schwarz inequality and Lemma \ref{lem:cheby}, 
\begin{align*}
&\frac{1}{\varphi_F^{*}(q)}\ \sumstar_{\chi \in \cap_{j=1}^{4} \mathcal S_j} F(\chi) \, \prod_{j=1}^{4}\mathbf{1}_{U_j}\big(\mathcal L(\chi\chi_j)\big)\\
&\qquad=\frac{1}{\varphi_F^{*}(q)}\ \sumstar_{\chi \pmod q} F(\chi) \, \prod_{j=1}^{4}\mathbf{1}_{U_j}\big(\mathcal L(\chi\chi_j)\big) +O\bigg(q^{-1/2} \Big(\#\{\chi \notin \cap_{j=1}^{4} \mathcal S_j\}\Big)^{1/2} \bigg)\\
&\qquad=\frac{1}{\varphi_F^{*}(q)}\ \sumstar_{\chi \pmod q} F(\chi) \, \prod_{j=1}^{4}\mathbf{1}_{U_j}\big(\mathcal L(\chi\chi_j)\big)+O(\Delta^{-1}).
\end{align*}
A similar argument applies to the main term on the right-hand side of \eqref{eq:ltop}, and so the sums in \eqref{eq:ltop} can be extended to all primitive characters modulo $q$ at the cost of an error term of size $O(\Delta^{-1})$.

Let $K(x)=(\frac{\sin \pi x}{\pi x} )^2$. Then $\widehat K(u)=\max\{0,1-|u|\}$. 
By Lemma \ref{lem:BS},  for $U=[a,b]$ and $\Delta>0$ we have $|\mathbf{1}_U(x)-W_{U,\Delta}(x)| \le 2$. Recall that $e_1=1+O(\varepsilon)$, so $$|\mathbf{1}_U(x)-W_{U,\Delta}(x)|^{e_1} \ll \Big( K\big(\Delta(x-a)\big)+K\big(\Delta(b-x)\big)\Big).$$ 
By H\"older's inequality and \eqref{upperbd} we have
\begin{align}\label{eq:smoothtosharp}
&\sumstar_{\chi \pmod q} |F(\chi)| \Big|\mathbf{1}_{\mathcal U_j}\big(\mathcal P(\chi\chi_j)\big)-W_{\mathcal U_j,\Delta}\big(\mathcal P(\chi\chi_j)\big)\Big|\nonumber\\
&\qquad \ll_\varepsilon q^{1/e_2}  \bigg(\,\, \sumstar_{\chi \pmod q} \Big( K\big(\Delta(\mathcal P(\chi\chi_j)-a_j)\big)+K\big(\Delta(b_j-\mathcal P(\chi\chi_j))\big)\Big)  \bigg)^{1/e_1}.
\end{align}

Now note that by Fourier inversion we have
\[
\frac{1}{\varphi^*(q)}\ \sumstar_{\chi \pmod q}K\!\big(\Delta(\mathcal P(\chi\chi_j)-\alpha)\big)
=\frac1\Delta\int_{-\Delta}^{\Delta}\Big(1-\frac{|u|}{\Delta}\Big)e(\alpha u)\,\Psi_q^{(j)}(u)\,du .
\]
Using Lemma \ref{lem:cf}, we have that $\Psi_q^{(j)}(u) \ll e^{-2 \pi^2 u^2}$, which gives that 
\begin{equation} \label{eq:errterm}
\sumstar_{\chi \pmod q} K\big(\Delta (\mathcal P(\chi \chi_j)-\alpha)\big) \ll \frac{q}{\Delta}.
\end{equation}
Hence, using \eqref{eq:smoothtosharp},
\eqref{eq:errterm} and \eqref{weight}, we can replace $\mathbf{1}_{\mathcal U_1}(\mathcal P(\chi\chi_1))$ by $W_{\mathcal U_1,\Delta}(\mathcal P(\chi\chi_1))$ in the main term in \eqref{eq:ltop} at the cost of an error term of size $O_\varepsilon(\Delta^{-1+\varepsilon})$. 
Similarly, we can replace all the $\mathbf{1}_{\mathcal U_j}(\mathcal P(\chi\chi_j))$ with $W_{\mathcal U_j,\Delta}(\mathcal P(\chi\chi_j))$ at the cost of an error term of the same size. Using Fourier inversion, \eqref{thm:mainresult2} and Lemma \ref{lem:BS}, we therefore get that the main term on the right-hand side of \eqref{eq:ltop} equals
\begin{equation}  \notag
\begin{split}
&\frac{1}{\varphi_F^{*}(q)}\ \sumstar_{\chi \pmod q} F(\chi)\prod_{j=1}^{4} W_{\mathcal U_j,\Delta}\big(\mathcal P(\chi\chi_j)\big)+O_\varepsilon(\Delta^{-1+\varepsilon})\\
&\qquad= \int_{\mathbb R^4} \Phi_q(\textbf{u})\prod_{j=1}^{4} \widehat {W_{\mathcal U_j, \Delta}}(u_j)  \, du_1du_2du_3du_4 +O_\varepsilon(\Delta^{-1+\varepsilon})\\
&\qquad = \int_{\mathbb R^4} e^{-2\pi^2\sum_{j=1}^4 u_j^2}\prod_{j=1}^{4} \widehat {W_{\mathcal U_j, \Delta}}(u_j) \, du_1du_2du_3du_4+O_\varepsilon(\Delta^{-1+\varepsilon}).
\end{split}
\end{equation}
Applying Plancherel and then using Lemma \ref{lem:BS}, we see that the main term above equals
\begin{align}\label{finale}
&\frac{1}{4\pi^2} \int_{\mathbb R^4} e^{-\sum_{j=1}^{4}x_j^2/2}\prod_{j=1}^{4} W_{\mathcal U_j,\Delta}(x_j) \, dx_1dx_2dx_3dx_4\nonumber\\
&\qquad = \frac{1}{4\pi^2} \int_{U_1 \times U_2\times U_3\times U_4} e^{-\sum_{j=1}^{4}x_j^2/2} \, dx_1dx_2dx_3dx_4+O_\varepsilon(\Delta^{-1+\varepsilon}).
\end{align}

To estimate the error term in \eqref{eq:ltop}, by arguing as above we have
\[
\frac{1}{\varphi^*(q)}\,\sumstar_{\chi \pmod q} \mathbf{1}_{U_{\delta,j}}\big(\mathcal P(\chi\chi_j)\big)=\frac{1}{\sqrt{2\pi}}\int_{U_{\delta,j}}e^{-x^2/2}dx+O_\varepsilon(\Delta^{-1+\varepsilon})\ll_\varepsilon \Delta^{-1+\varepsilon}. 
\]
Combining this with \eqref{finale}, we obtain the theorem.

\section{Proof of Corollary \ref{cor:simul large vals}}
  \label{cor_proof}
Fix $c>0$; we prove the first statement, since the second one is very similar. Put $a:=c\sqrt2+1$ and
$U_1=\dots=U_4=[a,a+1]$. 
Let
$$\mathcal T:=\Big\{\chi\pmod q, \chi \neq \chi_0:\mathcal L(\chi\chi_j)\in[a,a+1],\ j=1,\dots,4\Big\}.$$
For each $\chi\in\mathcal T$ and each $j$,
$\log|L(\half,\chi\chi_j)|\ge a\sqrt{\half\log\log q}=(c\sqrt2+1)\sqrt{\half\log\log q}
>c\sqrt{\log\log q}$, so we have
\begin{equation*}
|L(\half,\chi\chi_j)|>\exp\Big(c\sqrt{\log\log q}\Big),\qquad j=1,\dots,4.
\end{equation*}

By Theorem \ref{thm:cltjoint},
\begin{equation*}
\mu_F(\mathcal{T})
=\frac{1}{4\pi^2}\int_{U_1\times U_2\times U_3\times U_4}e^{-\sum x_j^2/2}\,dx_1dx_2dx_3dx_4
+O_\varepsilon\big((\log\log q)^{-1/2+\varepsilon}\big).
\end{equation*}
The main term equals
$\kappa:=\big(\tfrac{1}{\sqrt{2\pi}}\int_a^{a+1}e^{-x^2/2}\,dx\big)^4>0$, which depends
only on $c$. Hence for $q$ large, we have 
\begin{equation*}
\mu_F(\mathcal T) \gg 1.
\end{equation*}
By the definition of $\mu_F$ and \eqref{weight} ($\varphi_F^*(q)\asymp q$), it follows that
\begin{equation}\label{eq:weighted}
\sum_{\chi\in\mathcal T}  F(\chi)=\varphi_F^*(q)\,\mu_F(\mathcal T)\gg q.
\end{equation}

By the Cauchy-Schwarz inequality,
\begin{equation}
\label{firstcs}
\bigg| \sum_{\chi\in\mathcal T} F(\chi)\bigg|
\le\bigg(\ \sumstar_{\chi \pmod q}|F(\chi)|^2\bigg)^{1/2}|\mathcal T|^{1/2}.
\end{equation}
Also, using 
 H\"older's inequality, together with \eqref{upperbd}, we get that
$$
\sumstar_{\chi \pmod q}|F(\chi)|^2
\le\prod_{j=1}^4\bigg(\ \sumstar_{\chi \pmod q}|L(\half,\chi\chi_j)M(\half,\chi\chi_j)|^{8}\bigg)^{1/4}
\ll q.
$$
Combining the above with \eqref{eq:weighted} and \eqref{firstcs}
shows that
$|\mathcal{T} |\gg_c q$, which proves the first statement.

\bibliographystyle{amsalpha}

\end{document}